\RequirePackage{fix-cm}
\documentclass[a4paper,12pt,twoside,nospthms]{svjour3}     

\smartqed  

\usepackage{bm}
\usepackage{graphicx}
\usepackage{amsmath,amsthm}
\usepackage{subfigure}
\usepackage{listings} 
\usepackage{amssymb}  
\usepackage{indentfirst}
\usepackage{multirow}
\usepackage{color}
\usepackage{hyperref}
\usepackage{xcolor}

\newtheorem{thm}{Theorem}[section]

\newtheorem{lem}[thm]{Lemma}

\newtheorem{rem}{Remark}[section]
\newtheorem{exm}{Example}[section]
\newtheorem{scm}{Scheme}

\usepackage{forest}
\forestset{
  */.style={
    delay+={append={[]},}
  },
  rooted tree/.style={
    for tree={
      grow'=90,
      parent anchor=center,
      child anchor=center,
      s sep=2.5pt,
      if level=0{
        baseline
      }{},
      delay={
        if content={*}{
          content=,
          append={[]}
        }{}
      }
    },
    before typesetting nodes={
      for tree={
        circle,
        fill,
        minimum width=3pt,
        inner sep=0pt,
        child anchor=center,
      },
    },
    before computing xy={
      for tree={
        l=5pt,
      }
    }
  }
}

\usepackage{xparse}
\DeclareDocumentCommand\rootedtree{o}{\Forest{rooted tree [#1]}}
\newcommand{\vertiii}[1]{{\left\vert\kern-0.25ex\left\vert\kern-0.25ex\left\vert #1  \right\vert\kern-0.25ex\right\vert\kern-0.25ex\right\vert}}

\newcommand{\lap}{\mathrm{\Delta}}

\journalname{Journal}

\RequirePackage[top=2.5cm,bottom=3cm,left=2.4cm,right=2.4cm,includehead,includefoot]{geometry}

\begin{document}

\title{An exponential-free Runge--Kutta framework for developing third-order unconditionally energy stable schemes for the Cahn--Hilliard equation}

%
\author{Haifeng Wang$^1$\and
Jingwei Sun$^1$\and
Hong Zhang$^{1,*}$ \and
Xu Qian$^{1}$ \and
Songhe Song$^{1}$
}

\institute{
$^*$ Corresponding author	\\
E-mail addresses: hf\_wang1031@163.com (H. Wang), zhanghnudt@163.com (H. Zhang), qianxu@nudt.edu.cn (X. Qian),  shsong@nudt.edu.cn (S. Song) \\
$^1$ Department of Mathematics, College of Science, National University of Defense Technology, Changsha, Hunan, 410073, China \\
}

\date{Received: date / Accepted: date}

\maketitle

\begin{abstract}
In this work, we develop a class of up to third-order energy-stable schemes for the Cahn--Hilliard equation. Building on Lawson's integrating factor Runge--Kutta method, which is widely used for stiff semilinear equations, we discuss its limitations, such as the inability to preserve the equilibrium state and the oversmoothing of interfacial layers in the solution's profile because of the exponential damping effects. To overcome this drawback, we approximate the exponential term using a class of sophisticated Taylor polynomials, leading to a novel Runge--Kutta framework called exponential-free Runge--Kutta. By incorporating stabilization techniques, we analyze the energy stability of the proposed schemes and demonstrate that they preserve the original energy dissipation without time-step restrictions. Furthermore, we perform an analysis of the linear stability and establish an error estimate in the $\ell^2$ norm. A series of numerical experiments validate the high-order accuracy, mass conservation, and energy dissipation of our schemes. 


\end{abstract}

\section{Introduction}

The Cahn--Hilliard (CH) equation, which belongs to the class of phase field models, was originally introduced by Cahn and Hilliard in \cite{cahn1958free} to describe phenomena of phase separation and coarsening in non-uniform systems, such as glasses, alloys, and fixed-temperature polymer blends. 
While the CH equation offers an improvement over the Allen--Cahn (AC) equation by ensuring mass conservation, it also introduces higher-order spatial derivatives that increase stiffness and pose challenges for numerical computation. As a result, despite its theoretical advantages, creating a high-order, efficient, and stable numerical scheme that can deal with the stiffness from these spatial derivatives for the CH equation while maintaining the original characteristics is crucial in practical applications.

In this work, we consider the CH equation with the following form:
\begin{equation}
	\label{eqn:chsys}
	\left\{
		\begin{aligned}
			& u_t = \lap(-\epsilon^2\lap u + f(u)), &&\mathbf{x}\in \Omega,\ t\in (0,T],\\
			& u(\mathbf{x},0) = u_0(\mathbf{x}), &&\mathbf{x}\in \partial\Omega,
		\end{aligned}
	\right.
\end{equation}
where $\Omega = \prod_{i=1}^d(a_i,b_i)$ is a rectangle when $d = 2$ or a cuboid when $d = 3$, and the unknown solution $u$ is subject to the periodic boundary condition that represents the difference between two phase proportions. The diffusion coefficient $\epsilon$ characterizes the interfacial thickness of the transition region. The nonlinear function $f(\cdot)$ is defined as the derivative of the double-well potential function:
\begin{equation*}
	F(u) := \frac{1}{4}(u^2-1)^2,\quad f(u) := F'(u) = u^3 - u.
\end{equation*}

As is well known, the CH equation \eqref{eqn:chsys} is the $H^{-1}$ gradient flow \cite{fife2000models} with respect to the Ginzburg-Landau free energy functional $E(u)$, i.e.,
\begin{equation*}
	u_t  = \lap(\frac{\delta E(u)}{\delta u}),
\end{equation*}
where $\frac{\delta E(u)}{\delta u}$ denotes the standard variational derivative, and energy is represented as:
\begin{equation}
	\label{eqn:energy}
	E(u) := \int_{\Omega} \left( \frac{\epsilon^2}{2}|\nabla u|^2 + F(u)\right)\mathrm{d}\mathbf{x},
\end{equation}
with the symbol $|\cdot|$ denotes the Euclidean norm.
Therefore, the CH equation inherently satisfies the energy dissipation law:
\begin{equation*}
	\frac{\mathrm{d}}{\mathrm{d}t}E(u) = (u_t,\frac{\delta E(u)}{\delta u}) = - \Vert(-\lap)^{-\frac12}\frac{\partial u}{\partial t}\Vert^2_{L^2}<0,\quad \forall t>0,
\end{equation*}
where $(\cdot,\cdot)$ and $\Vert\cdot\Vert_{L^2}$ are the standard $L^2$ inner product and norm, respectively, and are defined as
\begin{equation*}
	(u,v) = \int_{\Omega}uv \mathrm{d} \mathbf{x},\quad \Vert u\Vert_{L^2} = (\int_{\Omega}|u|^2 \mathrm{d} \mathbf{x})^{\frac{1}{2}},\quad \forall u,v\in L^2(\Omega).
\end{equation*}
In addition, the CH equation is mass conservative along the time evolution, this is
\begin{equation}
	\int_{\Omega}u(\mathbf{x},t) \mathrm{d} \mathbf{x} = \int_{\Omega}u(\mathbf{x},0) \mathrm{d} \mathbf{x}, \quad t> 0,
\end{equation}
under a periodic boundary condition.

In the past two decades, there has been significant interest in developing numerical methods to solve the CH models, where the discrete energy dissipation law has garnered considerable attention in numerical analysis. In the early stage, researchers paid more attention to designing implicit methods to ensure numerical stability when solving the CH equation. In 2001, Furihata \cite{furihata2001stable} constructed a conservative Crank--Nicolson scheme to solve the one-dimensional CH equation and proved the inherit the properties of the equation such as the mass conservation and energy decrease under time step constrain. By separating the energy into the difference of two convex functional, Eyre \cite{eyre1998unconditionally,eyre1998unconditionally_ch} pioneeringly constructed the convex splitting scheme in 2009 and revealed its unconditional energy stability. Subsequently, Wise et al. \cite{wise2010unconditionally} adopted the convex splitting scheme to invest the Cahn--Hilliard-Hele--Shaw system and analyzed both the energy stability and convergence of the proposed scheme. More recently, combining the idea of the convex splitting scheme, Cheng et al. \cite{cheng2022third} proposed a third-order backward differentiation formula (BDF)-type scheme for the CH equation. They proved its energy stability with respect to a modified energy and demonstrated that the scheme guarantees the original energy is uniform boundedness. In addition, various numerical research has also focused on other phase field models, such as the thin film epitaxial model\cite{qiao2011adaptive}, and the phase field crystal model \cite{wise2009energy}. However, solving implicit schemes requires using Newton's solver at each time step, making it time-consuming and challenging for long-term simulations.

Since the CH equation requires long-term simulations to reach a steady state, it is essential to design efficient scheme to improve computational performance. In contrast to the traditional method \cite{shen2010numerical} that aims to preserve the dissipation of the original energy, the invariant energy quantization (IEQ) \cite{guillen2013linear,yang2016linear} and the scalar auxiliary variable (SAV) approach \cite{shen2019new,shen2018convergence} were proposed to develop unconditionally energy dissipation schemes with respect to a modified energy in the gradient flow problem and enable linear solvability in each time step. From the other aspect, the stabilization technique\cite{he2007large,xu2019stability} has gained significant attention in recent years and achieved remarkable success by combining it with various schemes. This technique explicitly handles the nonlinear terms and introduces a stabilization term to avoid time-step constraints, which eliminates the need for nonlinear iterations. Although this technique offers advantages, it also introduces a dispersion error in the solution \cite{xu2019stability,lee2019effective}. Therefore, low-order schemes usually cannot provide reliable solutions with a large time step, and high-order schemes are necessary for accurate simulation. Utilizing the stabilized exponential time difference (ETD) Runge--Kutta (RK) scheme, Li \cite{li2019convergence} conducted an error analysis for the first- and second-order schemes in solving the CH equation. They further proved the boundedness of the numerical solution under moderate constraints on the time step and spatial mesh size. After that, Fu et al.\cite{fu2022energy} provided an energy stability analysis for the ETDRK scheme with respect to the general phase-field models with the Lipschitz assumption. 
For the classical implicit-explicit Runge-Kutta (IMEX RK) scheme, Fu et al.\cite{fu2022unconditionally} also presented a general proof to show that the scheme can unconditionally preserve the energy dissipation law with stabilization techniques and certain RK coefficient constraints.
More information and applications of the stability technique can be found in \cite{duchemin2014explicit,li2022stability,tang2016implicit,du2019maximum} and the references therein.

The objective of this work is to construct a class of up to third-order, unconditional energy stable scheme for the CH equation.
Due to the strong stiffness of the perturbed biharmonic operator and the strong nonlinearity induced by the potential, solving the CH equation presents inherent numerical difficulties. 
There are many methods for nonlinear stiff problems, and one of the effective schemes is Lawson's integrating factor Runge--Kutta (IFRK) method \cite{lawson1967generalized,isherwood2018strong}.
However, IFRK methods also have their limitations, and one notable drawback is that they introduce the exponential damping effect\cite{hou1994removing,leo1998diffuseactamat}. 
Although this phenomenon can help suppress numerical instability, it can also oversmooth the numerical solution, which causes deviations from the solution of the original problem and makes it unreliable. In recent years, significant research \cite{du2019third,huang2018bound} efforts have been dedicated to improving the integrating factor Runge--Kutta (IFRK) method. In addition, our research group focuses on improving the IFRK method to preserve the maximum principles of AC-type equations and eliminate the time delay effects introduced by stabilization techniques \cite{zhang2022explicit}. 
Consistent with these endeavors, the objective of this work is to modify the IFRK scheme and develop new schemes that eliminate the exponential damping effect as well as preserve the energy dissipation of the CH equation. 
The contribution of our study can be summarized as follows:

\begin{itemize}
    \item By employing a Taylor-polynomial approximation to approximate the exponential function, we devise a family of exponential-free Runge--Kutta (EFRK) schemes that eliminate the exponential damping phenomenon in the IFRK method while preserving the steady state without destroying the convergence.
    \item We establish a unified framework to check whether the proposed EFRK method unconditionally preserves the original energy dissipation law. Meanwhile, we provide a class of up to third-order Runge--Kutta Butcher tableaux that satisfy this framework.
\end{itemize}

The remainder of the paper is organized as follows. In Section \ref{sec:2}, some preliminary information on the Fourier spectral collocation method is presented to obtain the spatial semi-discretization system of the CH equation. In Section \ref{sec:3}, we discuss the exponential damping phenomenon of the IFRK method and propose the EFRK framework that fixes this drawback, which can preserve the equilibrium states of the original system even with arbitrarily large time steps. Then, the linear stability and energy stability of the proposed scheme are rigorously proven in Section \ref{sec:4}, followed by error estimates derived in Section \ref{sec:5}. Section \ref{sec:6} presents various numerical experiments to test the accuracy and demonstrate the stability for long-time simulations. We finally provide the conclusions in Section \ref{sec:7}.

\section{Fourier spectral collocation approximations}
\label{sec:2}

In this section, we introduce the spectral collocation method for the spatial approximations of the CH equation.  Focusing on the one-dimensional (1D) interval $(a_1,b_1)$ with periodic boundary conditions. Let $N_1$ represent an even integer and denote $h_1 = \frac{b_1-a_1}{N_1}$; then, the interval can be uniformly partitioned as $\Omega_h = \{x_j|x_j = a_1 + jh_1, j= 0,1,\dots,N_1-1\}$. The grid function space with the periodic boundary condition defined on $\Omega_h$ is defined as $\mathcal{M}_{N_1}$, i.e.,
\begin{equation*}
  \mathcal{M}_{N_1} = \{\bm{v}|\bm{v} = [v(x_0),\dots,v(x_{N_1-1})]^T,x_j\in \Omega_h\}\in \mathbb{R}^{N_1}.
\end{equation*}

As described in \cite{shen2011spectral}, the spectral collocation method utilizes trigonometric basis, which are given by
\begin{equation*}
	\psi_j(x) = \frac{1}{N_1}\sum_{l=-N_1/2}^{N_1/2}\frac{1}{c_l}e^{\mathrm{i}\mu_1l(x-x_j)},
\end{equation*}
where $c_l = 1$ for $|l| < N_1/2$, $c_l = 2$ for $l= \pm N_1/2$, and $\mu_1 = \frac{2\pi}{b_1-a_1}$. Therefore, we obtain the corresponding interpolation function space $S_{N_1} = \text{span}\{\psi_j(x), j=0,1,\dots,N_1-1\}$, 
and the interpolation operator $I_{N_1} : C(\bar{\Omega}) \rightarrow S_{N_1}$ can defined as:
\begin{equation*}
	I_{N_1}u(x) = \sum_{j=0}^{N_1-1}u_j\psi_j(x) = \sum_{l=-N_1/2}^{N_1/2}\hat{u}_le^{\mathrm{i}\mu_1l(x-a_1)},\ \text{where}\
	\hat{u}_l = \frac{1}{N_1c_l}\sum_{j=0}^{N_1-1}u_je^{-\mathrm{i}\mu_1l(x_j-a_1)}.
\end{equation*}

Applying the 1D Laplacian operator $\lap$ to the interpolated function $I_{N_1}u(x)$ and utilizing the periodicity of trigonometric functions, we obtain
\begin{equation*}
	\begin{aligned}
	\lap (I_{N_1}u(x_i)) &= \sum_{l=-N_1/2}^{N_1/2} (\mathrm{i}\mu_1l)^2\left(\frac{1}{N_1c_l}\sum_{j=0}^{N_1-1}u_je^{-\mathrm{i}\mu_1l(x_j-a_1)}\right)e^{\mathrm{i}\mu_1l(x_i-a_1)}\\
						   &= \frac{1}{N_1}\sum_{l=0}^{N_1-1}\lambda_l^2\left(\sum_{j=0}^{N_1-1}u_je^{-\mathrm{i}\mu_1l(x_j-a_1)}\right)e^{\mathrm{i}\mu_1l(x_i-a_1)},\\
	\end{aligned}
\end{equation*}
where
\begin{equation*}
	\lambda_l^2 = 
	\left\{
		\begin{aligned}
			&(\mathrm{i}\mu_1l)^2,\quad &&0\leq l < N_1/2,\\
			&\frac{1}{2}(\mathrm{i}\mu_1l)^2 + \frac{1}{2}(\mathrm{i}\mu_1(l-N_1))^2, && l = N_1/2,\\
			&(\mathrm{i}\mu_1(l-N_1))^2,\quad && N_1/2<l\leq N_1-1.
		\end{aligned}
	\right.
\end{equation*}

Consequently, for any grid function $\bm{u} \in \mathcal{M}_{N_1}$, we can define the discrete Laplacian as $D_1^{(2)} = F_{N_1}^{-1}\Lambda_1^{(2)}F_{N_1}$, where $\Lambda_{1}^{(2)} = \text{diag}(\lambda_l^2)$ represents the eigenvalue of the Laplacian, and $F_{N_1},\ F_{N_1}^{-1}\in \mathbb{R}^{N_1\times N_1}$ denote the discrete Fourier transform and inverse discrete Fourier transform with the elements $(F_{N_1})_{l,j} = \frac{1}{N_1}e^{-\frac{\mathrm{i}2\pi lj}{N_1}}$, and $(F_{N_1}^{-1})_{j,l} = e^{\frac{\mathrm{i}2\pi lj}{N_1}}$.

For the case of the $d$-dimensional ($d\leq 3$) hyperrectangle domain $\Omega = \Pi_{k=1}^d(a_k,b_k)$, choose an even number $N_k$ as the grid number along the $k$th-dimension. We set $N = \Pi_{k=1}^{d}N_k$ as the total number of grid points. In addition, we denote $h_k = \frac{b_k-a_k}{N_k}$ as the grid size for each dimension and define $h = \max\limits_{k}\{h_k\}$, $\mu_k = \frac{2\pi}{b_k - a_k}$. For convenience, the $d$-dimensional grid function $\bm{u}$ is reshaped into a vector in the order of the first, second, and third dimension, which belongs to $\mathcal{M}_N$ and is equipped with the $\ell^2$-inner product, $\ell^2$-norm, and $\ell^\infty$-norm defined by
\begin{equation*}
	\langle \bm{u},\bm{v}\rangle : =  \Pi_{k=1}^{d}h_k\sum_{i=0}^{N-1}u_iv_i,\quad \Vert \bm{u}\Vert_{\ell^2}^2 = \langle\bm{u},\bm{u}\rangle,\quad \Vert \bm u\Vert_{\ell^{\infty}} = \max_{i}|u_i|,\quad \forall \bm{u},\bm{v}\in \mathcal{M}_N.
\end{equation*}

The discrete Laplacian in multi-dimensions can be implemented using the Kronecker products as 
\begin{equation*}
	\begin{aligned}
		&\text{1D:}\ \lap_{1,N} := D_1^{(2)},\\
		&\text{2D:}\ \lap_{2,N} := (F_{N_2}^{-1}\otimes F_{N_1}^{-1})(I_2\otimes \Lambda_1^{(2)} + \Lambda_2^{(2)}\otimes I_1 )(F_{N_2}\otimes F_{N_1}),\\
		&\text{3D:}\ \lap_{3,N} := (F_{N_3}^{-1}\otimes F_{N_2}^{-1}\otimes F_{N_1}^{-1})(I_3\otimes I_2\otimes \Lambda_1^{(2)} + I_3 \otimes\Lambda_2^{(2)}\otimes I_1+ \Lambda_3^{(2)}\otimes I_2 \otimes I_1)(F_{N_3}\otimes F_{N_2}\otimes F_{N_1}),\\
	\end{aligned}
\end{equation*}
where $\Lambda_k^{(2)}$ and $I_k$ are the corresponding eigenvalue and identity matrices, respectively, in the $k$th-dimension. To simplify the notation, we omit the dimension subscript and use $\lap_N$ to represent the discrete Laplacian operator in the remainder of this paper. 

In addition, the discrete Fourier transformation of $\bm{u}$ is denoted as
\begin{equation*}
\hat{\bm u} = 
\left\{
		\begin{aligned}
			& F_{N_1} \bm{u},&&\text{1D,}\\
			&(F_{N_2}\otimes F_{N_1})\bm{u},&&\text{2D,}\\
			&(F_{N_3}\otimes F_{N_2}\otimes F_{N_1})\bm{u},&&\text{3D.}\\
		\end{aligned}
	\right.
\end{equation*}
For any $\bm u, \bm v\in \ell^2(\mathcal{M}_N):=\{\bm u|\bm u\in \mathcal{M}_N, ||\bm u||_{\ell^2} \leq \infty \}$, Parseval's theorem gives
\begin{equation*}
	\langle\bm u,\bm v\rangle = |\Omega|({\hat{\bm v}}^{H}\hat{\bm u}),
\end{equation*}
where $|\Omega| = \Pi_{k=1}^d(b_k-a_k)$ denotes the area of the domain $\Omega$.

By noticing the mass conservation of the CH equation, we assume the mean of the solution $\bm u$ is zero without loss of generality. Consequently, in our analysis, we only consider the grid function with zero mean, which is denoted as $\mathcal{M}_N^0$,
\begin{equation*}
	\mathcal{M}_N^0 = \left\{\bm v\in \mathcal{M}_N|\langle \bm v,\bm 1\rangle = 0 \right\} = \left\{\bm v\in \mathcal{M}_N| \hat{v}_0 = 0\right\}.
\end{equation*}

Based on the above description, the semi-discretization of the CH equation on a $d$-dimensional domain is to find a function $\bm u: [0,\infty) \rightarrow \mathcal{M}_N^0$ such that
\begin{equation}
	\label{eqn:semi-disc}
	\left\{
		\begin{aligned}
			&\bm{u}_t = \lap_N(-\epsilon^2\lap_N\bm{u}+f(\bm{u})),\quad t\in(0,T],\\
			&\bm{u}(0) = \bm u_0.
		\end{aligned}
	\right.
\end{equation}
where $\bm u_0 \in \mathcal{M}_N^0$ is given by initial condition. 

\begin{rem}
	Since the eigenvalues of the discrete Laplacian matrix are non-positive, it holds that
	\begin{enumerate}
		\item[(i)] the matrix $\lap_N$ is self-adjoint and negative semi-definite on $\mathcal{M}_N$.
		\item[(ii)] the matrix $\lap_N$ is self-adjoint and negative definite on $\mathcal{M}_N^0$; thus, it is invertible on $\mathcal{M}_N^0$.
	\end{enumerate}
\end{rem}

By approximating the energy functional as $E(\bm u) = -\frac{\epsilon^2}{2}\langle \bm u,\lap_N \bm u\rangle + \langle F(\bm u), \bm 1 \rangle$, the semi-discrete schemes admit the energy dissipation law, i.e.,
\begin{equation*}
	\frac{\mathrm{d} E(\bm{u}(t))}{\mathrm{d} t} = \langle -\epsilon^2\lap_N\bm{u} + f(\bm{u}) ,\bm{u}_t\rangle =  \langle\lap_N^{-1} \bm u_t, \bm u_t\rangle\ \leq 0.
\end{equation*}

\section{Exponential-free Runge--Kutta framework}
\label{sec:3}

In this section, we begin by providing a concise review of Lawson's integrating factor Runge--Kutta (IFRK) method. Following this, we delve into a discussion on its exponential damping effects, which cause this method to fail in preserving equilibrium states and to oversmooth the solution at the interfacial layer. Subsequently, we propose a strategy to overcome this limitation and develop a new Runge--Kutta framework.

\subsection{Equilibrium-preserving Taylor-polynomial approximations}

For the sake of stability, we introduce a positive constant $\kappa \geq 0$ and add the term $\kappa \lap_N( \bm u - \bm u)$ on the right hand of the semi-discrete CH system \eqref{eqn:semi-disc}, yielding:
\begin{equation}
	\label{eqn:stab-reac-diff}
	\bm{u}_t = L_\kappa \bm{u} + N_\kappa(\bm u),
\end{equation}
where $L_\kappa = \lap_N(-\epsilon^2\lap_N + \kappa \bm u)$ and $N_\kappa = \lap_N f(u) - \kappa \lap_N u$. Using Lawson transformation \cite{lawson1967generalized} for the unknown $\bm{u}(t)$, i.e., $\bm{v}(t) = e^{-t L_\kappa} \bm{u}(t)$, we obtain an equivalent system of \eqref{eqn:stab-reac-diff} in the following form:
\begin{equation}
\label{eqn:lawsonform}
\left\{
	\begin{aligned}
		\bm{v}_t &= e^{- t L_\kappa} N_\kappa(e^{t L_\kappa}\bm v),\\
		\bm{v}(0)&= \bm u_0(\bm{x}).
	\end{aligned}
\right.	
\end{equation}

Let $\tau$ represent the time step. By employing an $s$-stage, $p$th-order explicit Runge--Kutta method to solve the reformulated system \eqref{eqn:lawsonform} and recovering the variable $\bm{v}$ back to $\bm{u}$ via $\bm u(t) = \mathrm{e}^{t L_\kappa} \bm v(t)$, we obtain the IFRK method:
\begin{equation}
	\label{eqn:sifrk}
	\left\{
		\begin{aligned}
			\bm{u}_{n,0} &= \bm{u}^n,\\
        	\bm{u}_{n,i} &= e^{c_i\tau L_\kappa}\big(\bm u^{n} + \tau\sum_{j=0}^{i-1}a_{i,j}e^{-c_j\tau L_\kappa} N_\kappa(\bm{u}_{n,j})\big),\quad i = 1,\dots,s,\\
			\bm{u}^{n+1} &= \bm{u}_{n,i},
		\end{aligned}
	\right.
\end{equation}
where $\bm{u}_{n,i}$ represents the approximation to $\bm{u}(t_{n,i})$ with $t_{n,i} = t_n + c_i\tau$ and $t_n = n \tau$. Butcher coefficients $a_{i,j}$ are constrained by certain accuracy requirements (Chapter 5 of \cite{gottlieb2011strong}). The abscissas $c_i$ are defined as $c_0 = 0$, $c_i = \sum_{j=0}^{i-1}a_{i,j}$ for $1 \leq i\leq s$, and $c_s = 1$ for consistency. 

Despite numerous advantages of the IFRK scheme in addressing the stiffness and preserving the strong stability for the hyperbolic conservation laws \cite{isherwood2018strong}, it has a notable limitation in that it cannot preserve the equilibrium states of the problem. Let $\bm u^* \neq \bm 0, \pm \bm 1$ denote a nontrivial equilibrium state to \eqref{eqn:stab-reac-diff}, i.e., $L_\kappa {\bm u}^* + {N_\kappa(\bm u^*)} = 0$. Denoting the discrete Fourier transform of $\bm u^*$ as $\hat{\bm u}^*$, and assuming that $\hat{\bm u}_{n,j} = \hat{\bm u}^*, j = 0, 1, \dots, i-1$, when $c_i \tau \neq 0$, we apply the discrete Fourier transform to the formulation \eqref{eqn:sifrk} and obtain
\begin{equation}
\label{eqn:exponential_damping}
\begin{aligned}
    \hat{\bm u}_{n, i} & = \mathrm{e}^{-{c}_i \tau \Lambda} ( \hat{\bm u}^n + \tau \sum_{j = 0}^{i-1} a_{i,j} \mathrm{e}^{c_j \tau \Lambda} [\underbrace{\widehat{ N_\kappa(\bm u_{n,j})} - \Lambda \hat{\bm u}_{n,j}}_{ = 0} + \Lambda \hat{\bm u}_{n, j} ] ) \\
    & = \mathrm{e}^{-{c}_i \tau \Lambda} \big(I + \tau \sum_{j = 0}^{i-1} a_{i,j}  \mathrm{e}^{c_j \tau \Lambda} \Lambda \big)\hat{\bm u}^* \\
    & \neq \hat{\bm u}^*, \quad  i \leq s,
    \end{aligned}
\end{equation}
where $\Lambda = \mathrm{diag}([\lambda_0, \dots, \lambda_{N-1}])$ is a diagonal matrix that consists of eigenvalues of the matrix $-L_\kappa = \epsilon^2 \lap_N^2 - \kappa \lap_N$ and the nonlinear term is evaluated in the physical space and subsequently transformed to the Fourier space.

Because $\mathrm{e}^{c_i \tau \Lambda}$ is not equal to $I + \tau  \sum_{j = 0}^{i-1} a_{i,j} \mathrm{e}^{c_j \tau \Lambda} \Lambda$ unless $c_i \tau \Lambda = 0$, the IFRK method \eqref{eqn:sifrk} will introduce unwanted exponential damping when $c_i > c_j$ for $i > j$ (or growth if $c_i < c_j$) to the Fourier modes of the equilibrium state, particularly those at high wavenumbers. Although the exponential damping effect tends to suppress instabilities such as those arising from an aliasing error or underresolution \cite{hou1994removing}, they can lead to a time-step-dependent smoothing of the high gradients associated with the interfacial layers \cite{leo1998diffuseactamat}. Consequently, this causes a discrepancy between numerical and analytical solutions at high wavenumbers.

Since our main focus is on the equilibrium states of the gradient flow system. To address this problem without destroying the convergence order, we utilize the $k$th-order Taylor polynomial $\phi_k(z) = 1 + z + \cdots + \frac{1}{k!}z^k $ to approximate the exponential function $e^{z}$. Specifically, we follow \cite{zhang2023third} and approximate the exponential function $e^{c_i \tau \Lambda}$ in \eqref{eqn:exponential_damping} using $\phi_i(c_i\tau \Lambda)$ to obtain
\begin{equation}
\label{eqn:equilirel}
    \hat{\bm u}_{n, i} = \frac{1}{\phi_i(c_i \tau \Lambda)} [I + \tau \sum_{j = 0}^{i-1} a_{i,j} \phi_j(c_j \tau \Lambda) \Lambda] \hat{\bm u}^* =: \hat{d}_i \hat{\bm u}^*, \quad i = 1, \dots, s.
\end{equation}

To preserve equilibrium states, we require that $\hat{d}_i \equiv I, i = 1, \dots, s$. Consider $s \leq 4$, the expansions of $\hat{d}_i$ can be expressed as
\begin{align} \nonumber
  \hat{d}_1 & = \frac{1}{\phi_1(c_1 \tau \Lambda)}[ I + c_1 \tau \Lambda], \\ \nonumber
  \hat{d}_2 & = \frac{1}{\phi_2(c_2 \tau \Lambda)}[ I + c_2 \tau \Lambda +  a_{2, 1} c_1 (\tau \Lambda)^2 ], \\ \nonumber
  \hat{d}_3 & = \frac{1}{\phi_3(c_3 \tau \Lambda)}\big(I + c_3 \tau \Lambda + (a_{3, 1} c_1 + a_{3, 2}c_2 )(\tau \Lambda)^2 +  a_{3, 2} a_{1} c_1 (\tau \Lambda)^3  \big), \\ \nonumber
  \hat{d}_4 & = \frac{1}{\phi_4(c_4 \tau \Lambda)}\big(I + c_4 \tau \Lambda + \sum_{j = 1}^3 a_{4, j} c_j (\tau \Lambda)^2  + [a_{4, 2} a_{2, 1} c_1 + a_{4, 3}(a_{3, 1} c_1 + a_{3, 2} c_2)] (\tau \Lambda)^3 + a_{4, 3} a_{3, 2} a_{2, 1} c_1 (\tau \Lambda)^4   \big).
  \end{align}

Combine with the order conditions of the RK scheme in Table \ref{tab:rootedtreeso}, we arrived at the following conclusion:
\begin{enumerate}
  \item RK(1, 1) satisfies $\hat{d}_1 \equiv I$;
  \item RK(2, 2) satisfies $\hat{d}_i \equiv I, i = 1, 2$;
  \item RK(3, 3) satisfies $\hat{d}_i \equiv I, i = 1, 3$. $\hat{d}_2 \equiv I$ requires the RK coefficients additional satisfy $a_{2, 1} c_1 = \frac{c_2^2}{2}$. In the literature, we find that Heun's method [RK(3, 3) in (\ref{eqn:rk123})] meets this condition;
  \item RK(4, 4) satisfies $\hat{d}_i \equiv I, i = 1, 4$. Ensuring that $\hat{d}_i \equiv I, i = 2, 3$ requires the RK coefficients further satisfy $a_{3,1}c_1 + a_{3,2}c_2 = \frac{c_3^2}{2}$, $a_{3,2}a_{2,1}c_1= \frac{c_3^3}{6}$, and $a_{2,1}c_1 = \frac{c_2^2}{2}$. Since the explicit RK(4,4) method has only ten parameters, and there are eleven constraints in total, this signifies that there is an insufficient degree of freedom for determining the parameters. Therefore, no four-stage, fourth-order explicit RK coefficients can satisfy these conditions. 
\end{enumerate}

\begin{table}[!h] 
\caption{Order conditions for the explicit Runge--Kutta methods.}
\label{tab:rootedtreeso}
\begin{center}
\begin{tabular}{|c|c|c|c|c|} \hline
Tree & Order 1 & Order 2 & Order 3  & Order 4  \\ \hline
 \rootedtree []
 & $b^T \bm 1 = 1$
 & $b^T \bm 1 = 1$
 & $b^T \bm 1 = 1$
 & $b^T \bm 1 = 1$ \\
\rootedtree [*]
& & $b^T c = \frac{1}{2}$
& $b^T c = \frac{1}{2}$
& $b^T c = \frac{1}{2}$ \\
\rootedtree [*, *]
& & & $b^T c^2 = \frac{1}{3}$
& $b^T c^2 = \frac{1}{3}$\\
\rootedtree [[*], *] & &
& $b^T A c = \frac{1}{6}$
& $b^T A c = \frac{1}{6}$  \\
\rootedtree [*, *, *] & &  & & $b^T c^3 = \frac{1}{4}$ \\
\rootedtree [*, [*],*] & & & & $b^T [c \cdot (A c)] = \frac{1}{8}$\\
\rootedtree [[*, *], ] & & & & $ b^T A c^2 = \frac{1}{12}$ \\
\rootedtree [[[*],*, *], *] & & & & $ b^T A^2 c = \frac{1}{24}$ \\\hline
	\end{tabular}
\end{center}
Here, $\bm 1 = [1, \dots, 1]^T \in \mathbb{R}^{s}$, $c = [c_0, c_1, \dots, c_{s-1}]^T$, $A = (a_{i,j})_{i, j = 0, \dots, s-1} \in \mathbb{R}^{s\times s}$, $b = [a_{s,0}, a_{s,1}, \dots, a_{s, s-1}]^T$, and powers of vectors are element-wise.
\end{table}

Some up to third-order underlying Runge-Kutta schemes that satisfy both the order condition and $\hat{d}_i \equiv I$ for $i = 1, \dots, s$ are presented below:
\begin{equation}\label{eqn:rk123}
    \begin{array}{ccc}
    \mathrm{RK(1, 1):}
	\begin{array}{c|c}
		0 & 0 \\ \hline
		1 & 1
	\end{array},  & 
  \mathrm{RK(2, 2):} 
  \begin{array}{c|cc}
    0 & 0 &0\\
    1 & 1 & 0\\ \hline
    1 & \frac{1}{2} & \frac{1}{2}
    \end{array}, &
    \mathrm{RK(3, 3):}
    \begin{array}
      {c|ccc}
      0 & 0 & 0 & 0 \\
      \frac{1}{3} & \frac{1}{3} & 0 & 0  \\
      \frac{2}{3} & 0 & \frac{2}{3}  & 0\\ \hline
      1 & \frac{1}{4} & 0 & \frac{3}{4}
    \end{array}. \\
     \text{Forward Euler scheme}\quad  & \text{Heun's second order scheme}\quad  & \text{Heun's third order scheme (p. 135 in \cite{hairer1993solving}) }
    \end{array}
\end{equation}

By replacing $\mathrm{e}^{-c_i \tau L_\kappa}$ with $\phi_i(-c_i\tau L_\kappa)$ in the IFRK framework \eqref{eqn:sifrk}, we obtain 
\begin{equation}
	\label{eqn:efrk}
	\left\{
		\begin{aligned}
			\bm{u}_{n,0} &= \bm{u}^n,\\
        	\bm{u}_{n,i} &= \left(\phi_i(-c_i\tau L_\kappa)\right)^{-1}\big(\bm u^n + \tau\sum_{j=0}^{i-1}a_{i,j} \phi_j(-c_j\tau L_\kappa) N_\kappa(\bm{u}_{n,j})\big),\quad i = 1,\dots,s,\\
			\bm{u}^{n+1} &= \bm{u}_{n,i}.
		\end{aligned}
	\right.
\end{equation}

Since the requirements $\hat{d}_i \equiv I$ for $i = 1, \dots, s$ guarantee the preservation of equilibrium states, and eliminates the exponential effects, the formulation \eqref{eqn:efrk} is referred to as the exponential-free Runge--Kutta (EFRK) framework. Some concrete EFRK schemes based on the underlying RK Butcher tableaux \eqref{eqn:rk123} are presented as below.

\begin{scm}[EFRK(1, 1)] \label{ex:erk11}
\begin{align}\label{eqn:erk11}
  \bm u^{n+1} & = \left(\phi_1(-\tau L_\kappa)\right)^{-1}\big(\bm u^n + \tau N_\kappa(\bm u^n)\big),
\end{align}
where $\phi_1(-\tau L_\kappa) = I - \tau L_\kappa$.
\end{scm}

The EFRK(1, 1) scheme is indeed the stabilization semi-implicit scheme investigated by He et al. \cite{he2007large} and Li et al. \cite{li2016characterizing}.

\begin{scm}[EFRK(2, 2)]\label{ex:erk22}
\begin{equation}\label{eqn:erk22}
\left\{
\begin{aligned}
 \bm u_{n, 1} & = \left(\phi_1(-\tau L_\kappa)\right)^{-1}\big(\bm u^n + \tau N_\kappa(\bm u^n)\big), \\
\bm u^{n+1} & = \left(\phi_2(-\tau L_\kappa)\right)^{-1}\big(\bm u^n + \frac{1}{2}\tau N_\kappa(\bm u_{n, 0}) + \frac{1}{2}\tau \phi_1(-\tau L_\kappa) N_\kappa(\bm u_{n, 1})\big),
\end{aligned}  \right.
\end{equation}
where $\phi_1(-\tau L_\kappa) = I - \tau L_\kappa$, $\phi_2(-\tau L_\kappa) = I - \tau L_\kappa + \frac{1}{2} (\tau L_\kappa)^2$.
\end{scm}

\begin{scm}[EFRK(3, 3)]\label{ex:erk33}
\begin{equation}
  \label{eqn:erk33}
  \left\{
  \begin{aligned}
    \bm u_{n, 1} & = \big(\phi_1(-\frac{1}{3} \tau L_\kappa)\big)^{-1}\big(\bm u^n + \frac{1}{3}\tau N_\kappa(\bm u_{n, 0})\big), \\
    \bm u_{n, 2} & = \big(\phi_2(-\frac{2}{3} \tau L_\kappa)\big)^{-1}\big(\bm u^n + \frac{2}{3}\tau \phi_1(-\tau L_\kappa) N_\kappa(\bm u_{n, 1})\big), \\
    \bm u_{n, 3} & = \big(\phi_3(-\tau L_\kappa)\big)^{-1}\big(\bm u^n + \frac{1}{4} \tau N_\kappa(\bm u_{n, 0}) + \frac{3}{4} \tau \phi_2(-\tau L_\kappa) N_\kappa(\bm u_{n, 2})\big),
  \end{aligned}
  \right.
\end{equation}
where 
\begin{align*}
&\phi_1(-\frac{1}{3}\tau L_\kappa) = I - \frac{1}{3}\tau L_\kappa,\\ 
&\phi_2(-\frac{2}{3}\tau L_\kappa) = I - \frac{2}{3}\tau L_\kappa +  \frac{1}{2}(\frac{2}{3} \tau L_\kappa)^2,\\ 
&\phi_3(-\tau L_\kappa) = I - \tau L_\kappa + \frac{1}{2} (\tau L_\kappa)^2 - \frac{1}{6} (\tau L_\kappa)^3.
\end{align*}
\end{scm}

Next, we demonstrate that the EFRK framework conserves the mass. We recall a useful lemma on matrix functions.
\begin{lem}\cite{higham2008functions}
	\label{lem:matrixFunc}
If $\phi$ is defined on the spectrum of $M\in \mathbb{C}^{n\times n}$, i.e., the values 
\begin{equation}
	\phi^{(j)}(\lambda_i),\quad 0\leq j\leq n_i-1,\quad 1\leq i\leq n
\end{equation}
exist, where $\{\lambda_i\}_{i=1}^n$ are the eigenvalues of $M$, and $n_i$ is the order of the largest Jordan block where $\lambda_i$ appears, we have
\begin{enumerate}
	\item $\phi(M)$ commutes with $M$; 
	\item $\phi(M^T) = \phi(M)^T$;
	\item the eigenvalues of $\phi(M)$ are $\{\phi(\lambda_i)\}_{i=1}^{n}$.
\end{enumerate}
\end{lem}


\begin{thm}
	The EFRK method \eqref{eqn:efrk} conserves the mass of the semi-discrete CH system, that is,
	\begin{equation*}
		\langle \bm u^{n+1} - \bm u^n,\bm 1\rangle = 0,\quad n\geq 0,
	\end{equation*}
	where $\bm 1 = [1,\dots,1]^T \in \mathbb{R}^N$.
\end{thm}
\begin{proof}
	Note that $\lap_N$ is symmetric and commutes with $\phi_j(-c_j\tau L_\kappa)$. Multiplying both sides of \eqref{eqn:efrk} with $\phi_i(-c_i\tau L_\kappa)$ and taking the $\ell^2$ inner product with $\bm 1$, we obtain
	\begin{equation}
		\label{eqn:mass1}
		\begin{aligned}
		\langle \phi_i(-c_i\tau L_\kappa) \bm u_{n,{i}}, \bm 1\rangle &=  \langle \bm u^n,\bm 1\rangle + \langle \tau \sum_{j=0}^{i-1} a_{i,j} \phi_j(-c_j\tau L_\kappa) N_\kappa(\bm u_{n,j}),\bm 1 \rangle\\
                & = \langle \bm u^n,\bm 1\rangle + \langle \tau \sum_{j=0}^{i-1} a_{i,j} \phi_j(-c_j\tau L_\kappa)[f(\bm u_{n,j}) - \kappa \bm u_{n, j}],\lap_N \bm 1 \rangle \\
                & = \langle \bm u^n, \bm 1\rangle, \quad i = 1,\dots,s.\\
		\end{aligned}
	\end{equation}
    Based on lemma \ref{lem:matrixFunc}, the functions $\phi_i(-c_i\tau L_\kappa)$ are symmetric. Then, it holds that
	\begin{equation}
		\label{eqn:mass2}
		\begin{aligned}
		  \langle \phi_i(-c_i\tau L_\kappa) \bm u_{n,i},\bm 1\rangle &= \langle \bm u_{n,i},\phi_i(-c_i\tau L_\kappa)\bm 1\rangle \\
		  & = \langle \bm u_{n,i},\bm 1\rangle - \langle \bm u_{n,i},\tau\sum_{j=0}^{i-1}a_{ij}\phi_j(-c_j\tau L_\kappa)L_\kappa \bm 1\rangle  \\
            & = \langle \bm u_{n, i}, \bm 1\rangle, \quad i = 1, \dots, s.
		\end{aligned}
	\end{equation}
	Substituting equation \eqref{eqn:mass2} into \eqref{eqn:mass1} yields
	\begin{equation*}
		  \langle \bm u_{n,i}- \bm u^n,1\rangle = 0,\quad i = 1,\dots,s.
	\end{equation*}
	This completes the proof because $\bm u^{n+1} = \bm u_{n, s}$.
\end{proof}



\section{Stability analysis}
\label{sec:4}

\subsection{Energy stability of EFRK scheme}
Because of the preservation of equilibrium states, we demonstrate that a certain class of first- to third-order EFRK schemes unconditionally decreases the energy of the CH equation when the stabilization parameter is sufficiently large. First, we have the following lemma, which will play an important role in our analysis of energy stability.

\begin{lem}\label{lem:sumomega0}
   Let $\bm u_{n,i}$ be the $i$th-stage solution ($i = 0, 1, \dots, s$) obtained by the stabilization EFRK schemes \eqref{eqn:efrk}; then, it holds that
  \begin{align}\label{eqn:sumomega0}
    L_\kappa \bm u_{n,i} + N_\kappa(\bm u_{n, i-1}) = \tau^{-1}\sum_{k = 0}^i \omega_{i,k} \bm u_{n,k}, \quad \text{with~} \sum_{k = 0}^i \omega_{i,k} = 0,
  \end{align}
  where $\omega_{i,k}$ depends on the RK coefficients $a_{i,j}$, $\tau$, and $L_\kappa$.
\end{lem}
\begin{proof}
Let us denote $\phi_i$ as abbreviations of $\phi_i(-c_i\tau L_\kappa), i = 0, 1, \dots, s$. We proceed with a proof by induction.

When $i = 1$, we have
\begin{align*}
  L_\kappa \bm u_{n, 1} + N_\kappa(\bm u_{n,0}) & = L_\kappa \bm u_{n, 1} + (a_{1, 0} \phi_0 \tau)^{-1} \big( \phi_1 \bm u_{n, 1} - \bm u_{n, 0}\big) \\
  & = (a_{1,0} \phi_0 \tau)^{-1}\big((a_{1, 0} \phi_0 \tau L_\kappa + \phi_1) \bm u_{n, 1} - \bm u_{n, 0}\big) \\
  & = \tau^{-1}(\omega_{1,1} \bm u_{n, 1} + \omega_{1,0} \bm u_{n, 0}),
\end{align*}
where $\omega_{1,1} = (a_{1,0} \phi_0 )^{-1}(a_{1,0} \phi_0 \tau L_\kappa + \phi_1)$, $\omega_{1,0} = -(a_{1, 0} \phi_0 )^{-1}$. Using the definition of $\phi_1$, it is clear that $\omega_{1,1} + \omega_{1,0} = 0$. 

Assume that result \eqref{eqn:sumomega0} holds for $i = 1, \dots, \ell-1$, with $\ell \leq s$; then, $N_\kappa(\bm u_{n, i-1}) = \tau^{-1}\sum_{k = 0}^{i} \omega_{i, k} \bm u_{n, k} - L_\kappa \bm u_{n, i}$.
For case $i = \ell$, we have
\begin{align*}
   & L_\kappa \bm u_{n, \ell} + N_\kappa(\bm u_{n, \ell-1}) \\
  =& L_\kappa \bm u_{n, \ell} + (a_{\ell,\ell-1} \phi_{\ell-1} \tau)^{-1}\big( \phi_\ell \bm u_{n, \ell} - \bm u_{n, 0} - \sum_{j = 0}^{\ell-2} a_{\ell,j} \phi_j \tau N_\kappa(\bm u_{n, j})\big) \\
  =& L_\kappa \bm u_{n, \ell} + (a_{\ell,\ell-1} \phi_{\ell-1} \tau)^{-1}\big( \phi_\ell \bm u_{n, \ell} - \bm u_{n, 0} - \sum_{j = 0}^{\ell-2} a_{\ell,j} \phi_j \tau (\tau^{-1} \sum_{k = 0}^{j+1} \omega_{j+1,k} \bm u_{n, k} + L_\kappa \bm u_{n,j+1}) \big), \quad \ell \leq s.
\end{align*}
Thus, $L_\kappa \bm u_{n, \ell} + N_\kappa(\bm u_{n, \ell-1})$ is decomposed into a combination of $\bm u_{n, k}$, and the summation of coefficients satisfies
\begin{align*}
  \sum_{k = 0}^{\ell} \omega_{\ell,k} &= \tau L_\kappa + (a_{\ell,\ell-1} \phi_{\ell-1} )^{-1}\big( \phi_\ell - I + \sum_{j = 0}^{\ell-2} a_{\ell,j} \phi_j \tau L_\kappa\big)\\ 
                                &= (a_{\ell,\ell-1} \phi_{\ell-1} )^{-1}\big(\phi_\ell - I + \sum_{j = 0}^{\ell-1} a_{\ell,j} \phi_j \tau L_\kappa\big) \\
                                &= 0.
\end{align*}
This completes the proof.
\end{proof}

\begin{lem}\label{lem:sumDelta}
The $i$th-stage solution $\bm u_{n, i}$ ($i = 1, \dots, s$) obtained by the stabilization EFRK method satisfies the following relationship:
  \begin{align}\label{eqn:sumDelta}
    L_\kappa \bm u_{n,i} + N_\kappa(\bm u_{n,i-1}) = \tau^{-1}\sum_{j = 1}^{i} \Delta_{i, j} (\bm u_{n, j} - \bm u_{n,j-1}), \quad \text{with~} \Delta_{i,j} = \sum_{k = j}^{i} \omega_{i,k}, j = 1, \dots, i,
  \end{align}
where $\omega_{i,k}$ are the expansion coefficients of $L_\kappa \bm u_{n, i} + N_\kappa(\bm u_{n, i-1})$ in \eqref{eqn:sumomega0}.
\end{lem}
\begin{proof}
By utilizing the Lemma \ref{lem:sumomega0} and noting that $\sum_{k = 1}^i \omega_{i,k} = - \omega_{i,0}$, we can obtain 
  \begin{align*}
    \tau\left(L_\kappa \bm u_{n,i} + N_\kappa (\bm u_{n,i-1}) \right) =& \sum_{k = 0}^i  \omega_{i,k} \bm u_{n, k} = \sum_{k = 1}^i \omega_{i,k} \bm u_{n, k} - \sum_{k = 1}^i \omega_{i,k}  \bm u_{n, 0}\\
     =& \sum_{k = 2}^i  \omega_{i,k} \bm u_{n, k} - \sum_{k = 2}^i \omega_{i,k}  \bm u_{n, 1} + \sum_{k = 1}^i \omega_{i,k} (\bm u_{n,1} - \bm u_{n, 0})\\
      &\cdots\\
     =& \omega_{i,i} (\bm u_{n, i} - \bm u_{n ,i-1}) + \dots + \sum_{k = 1}^{i} \omega_{i,k} (\bm u_{n, 1} - \bm u_{n, 0}) \\
     =& \sum_{j = 1}^{i} \Delta_{i, j} (\bm u_{n, j} - \bm u_{n,j-1}).
  \end{align*}
This completes the proof.
\end{proof}

\begin{lem}
\label{lem:energy_decrease}
  Assume that the stabilization parameter satisfies $\kappa \geq \max\limits_{|\xi| \leq \beta} \frac{1}{2}|f'(\xi)|$, where $\beta = \max_{n, i}\{\| \bm u_{n, i}\|_{\ell^\infty} \}$, the two adjacent stage solutions, as derived from the stabilization EFRK method, satisfy the following relationship:
  \begin{align*}
    E(\bm u_{n, i}) - E(\bm u_{n,i-1}) & \leq \langle \bm u_{n, i} - \bm u_{n, i-1}, \lap_N^{-1} \big(L_\kappa \bm u_{n, i} + N_\kappa(\bm u_{n, i-1})\big) \rangle \\
    & = \langle \bm u_{n, i} - \bm u_{n, i-1}, (\tau \lap_N)^{-1} \sum_{j = 1}^i \Delta_{i,j}(\bm u_{n, j} - \bm u_{n, j-1}) \rangle, \quad i = 1, \dots, s.
  \end{align*}
where $\Delta_{i,j}$ is defined in \eqref{eqn:sumDelta}.
\end{lem}
\begin{proof}

For any $\bm u, \bm v \in \mathbb{R}^N$ with $\|\bm u\|_{\ell^\infty} \leq \beta$ and $\|\bm v\|_{\ell^\infty} \leq \beta$, when $\kappa \geq \max\limits_{|\xi| \leq \beta} \frac{1}{2} |f'(\xi)|$,
applying a Taylor expansion to the difference of nonlinear parts of the energy, and using the identity $\lap_N^{-1} {N}_{\kappa}(\bm u) = f(\bm u) - \kappa \bm u$  yield
\begin{equation}\label{eqn:energyN}
\begin{aligned}
  \langle F(\bm v) - F(\bm u), 1\rangle & = \langle  \bm v - \bm u, f(\bm u)\rangle + \langle \bm v - \bm u, \frac{1}{2}f'(\bm \xi)\cdot (\bm v - \bm u)\rangle \\
  & \leq \langle \bm v - \bm u, f(\bm u)\rangle  +  \kappa  \langle \bm v - \bm u, \bm v - \bm u \rangle \\
  & = \langle \bm v - \bm u, \lap_N^{-1} {N}_{\kappa}(\bm u) \rangle +  \kappa (\bm v - \bm u,  \bm u) + \kappa \langle\bm  v - \bm u, \bm v - \bm u\rangle.
\end{aligned}
\end{equation}
For the difference of the linear parts, it holds that
\begin{equation}
		\label{eqn:energyL}
		\begin{aligned}
			-\frac{\epsilon^2}{2}(\langle \bm v, \lap_N \bm v\rangle - \langle \bm u,\lap_N \bm u\rangle) &= -\langle \bm v-\bm u,\epsilon^2\lap_N \bm v\rangle + \frac{1}{2}\langle \bm v-\bm u,\epsilon^2\lap_N(\bm v-\bm u)\rangle\\
			&= \langle \bm v-\bm u,(-\epsilon^2\lap_N+\kappa)\bm v \rangle - \kappa\langle \bm v- \bm u, \bm v \rangle + \frac{1}{2}\langle \bm v-\bm u,\epsilon^2\lap_N(\bm v-\bm u)\rangle \\
                & = \langle \bm v - \bm u, \lap_N^{-1} L_\kappa \bm v\rangle - \kappa\langle \bm v- \bm u, \bm v \rangle + \frac{1}{2}\langle \bm v-\bm u,\epsilon^2\lap_N(\bm v-\bm u)\rangle.
		\end{aligned}
	\end{equation}
Adding \eqref{eqn:energyN} and \eqref{eqn:energyL} and using the negative semi-definite property of $\lap_N$ yields
\begin{equation}\label{eqn:energyinequal}
\begin{aligned}
  E(\bm v) - E(\bm u) & \leq \langle \bm v - \bm u, \lap_N^{-1}[{L}_{\kappa} \bm v + {N}_{\kappa}(\bm u)] \rangle  + \frac{1}{2}\langle \bm v - \bm u, \epsilon^2 \lap_N (\bm v - \bm u)\rangle \\
  & \leq \langle \bm v - \bm u, \lap_N^{-1} [{L}_{\kappa}\bm  v + {N}_{\kappa}(\bm u)] \rangle.
\end{aligned}
\end{equation}
Letting $\bm v = \bm u_{n, i}$, $\bm u = \bm u_{n, i-1}$, and substituting the equality \eqref{eqn:sumDelta} into \eqref{eqn:energyinequal} give
\begin{align*}
  E(\bm u_{n, i}) - E(\bm u_{n, i-1}) & \leq \langle \bm u_{n, i} - \bm u_{n, i-1}, \lap_N^{-1}[L_\kappa \bm u_{n,i} + N_\kappa(\bm u_{n, i-1})])\rangle \\
  & = \langle \bm u_{n, i} - \bm u_{n, i-1}, (\tau\lap_N)^{-1} \sum_{j = 1}^{i} \Delta_{i,j}(\bm u_{n,j} - \bm u_{n, j-1})\rangle.
\end{align*}
This completes the proof.
\end{proof}

\begin{table}[!h] 
	\caption{Energy stability matrices for an EFRK($s, p$) scheme.\label{tab:escerk}}
\begin{center}
\begin{tabular}{|c|c|c|c|c|} \hline
Stages & 1 & 2 & 3 & $s$ \\ \hline
& $\Delta_{1, 1}$ & $\Delta_{1, 1} - \frac{1}{2} \Delta_{2, 1}$ & $\Delta_{1, 1} - \frac{1}{2} \Delta_{2, 1} - \frac{1}{2} \Delta_{3, 1}$ & $\Delta_{1, 1} - \frac{1}{2} \sum_{j = 2}^s \Delta_{j, 1}$ \\
&  & $\Delta_{2, 1}$ & $\Delta_{2, 1}$ & $\Delta_{2, 1}$ \\
Energy &  & $\Delta_{2, 2} - \frac{1}{2} \Delta_{2, 1}$ & $\Delta_{2, 2} - \frac{1}{2}\Delta_{2, 1} - \frac{1}{2} \Delta_{3, 2}$ & $\Delta_{2, 2} - \frac{1}{2} \Delta_{2, 1} - \frac{1}{2} \sum_{j = 3}^s \Delta_{j, 2}$ \\
stability&  & & $\Delta_{3, 1}$ & $\Delta_{3, 1}$ \\
matrices&  & & $\Delta_{3, 2}$ & $\Delta_{3, 2}$ \\
&  & & $\Delta_{3, 3} - \frac{1}{2} \Delta_{3, 1} - \frac{1}{2} \Delta_{3, 2}$ & $\Delta_{3, 3} - \frac{1}{2} \sum_{j = 1}^2 \Delta_{3, j}$ \\
&  & & & $\Delta_{i,j}, i = 1, \dots, s; j = 1, \dots, i-1$ \\
&  & & & $\Delta_{i,i} - \frac{1}{2} \sum_{j = 1}^{i-1} \Delta_{i,j} - \frac{1}{2} \sum_{j = i+1}^s \Delta_{j,i}, i = 1, \dots, s$ \\ \hline
	\end{tabular}
\end{center}
\end{table}

\begin{thm}
    Under the same assumption of the stabilization parameter in Lemma \ref{lem:energy_decrease} and supposing that the matrices in Table \ref{tab:escerk} are positive semi-definite, the stabilization EFRK($s, p$) method is unconditionally energy stable for any $\tau > 0$.
\end{thm}

\begin{proof}
By decomposing the energy difference into a combination of quadratic forms, we obtain
\begin{align*}
   & E(\bm u_{n,s}) - E(\bm u_{n,0}) = \sum_{i = 1}^s E(\bm u_{n, i}) - E(\bm u_{n, i-1}) \\
   \leq & \sum_{i = 1}^s \langle \bm u_{n, i} - \bm u_{n, i-1}, (\tau\lap_N)^{-1} \sum_{j = 1}^i \Delta_{i,j}(\bm u_{n, j} - \bm u_{n, j-1}) \rangle \\
  = & \sum_{i = 1}^s \langle \bm u_{n, i} - \bm u_{n, i-1}, (\tau\lap_N)^{-1} (\Delta_{i,i} - \frac{1}{2} \sum_{j = 1}^{i-1} \Delta_{i,j} - \frac{1}{2} \sum_{j = i+1}^{s} \Delta_{j, i} )  (\bm u_{n, i} - \bm u_{n, i-1})\rangle \\
    & + \sum_{i = 1}^s \langle \bm u_{n, i} - \bm u_{n, i-1}, \frac{1}{2} (\tau\lap_N)^{-1} \sum_{j = 1}^{i-1} \Delta_{i,j}  (\bm u_{n, i} - \bm u_{n, i-1})\rangle\\
    & + \sum_{i = 1}^s \langle \bm u_{n, i} - \bm u_{n, i-1}, \frac{1}{2}(\tau\lap_N)^{-1} \sum_{j = i+1}^{s} \Delta_{j, i}   (\bm u_{n, i} - \bm u_{n, i-1})\rangle \\
    & + \sum_{i = 1}^s \langle \bm u_{n, i} - \bm u_{n, i-1}, (\tau\lap_N)^{-1} \sum_{j = 1}^{i-1} \Delta_{i,j}(\bm u_{n, j} - \bm u_{n, j-1}) \rangle \\
    = & \sum_{i = 1}^s \langle \bm u_{n, i} - \bm u_{n, i-1}, (\tau\lap_N)^{-1} (\Delta_{i,i} - \frac{1}{2} \sum_{j = 1}^{i-1} \Delta_{i,j} - \frac{1}{2} \sum_{j = i+1}^{s} \Delta_{j, i} )  (\bm u_{n, i} - \bm u_{n, i-1})\rangle \\
    & + \sum_{i = 1}^s\sum_{j = 1}^{i-1} \langle \bm u_{n, i} - \bm u_{n, i-1}, \frac{1}{2}(\tau\lap_N)^{-1}  \Delta_{i,j}  (\bm u_{n, i} - \bm u_{n, i-1} + \bm u_{n,j} - \bm u_{n,j-1})\rangle \\
    & + \sum_{i = 1}^s \sum_{j = 1}^{i-1} \langle u_{n, i} - \bm u_{n, i-1} + \bm u_{n, j} - \bm u_{n, j-1}, \frac{1}{2}(\tau\lap_N)^{-1} \Delta_{i,j}   (\bm u_{n, j} - \bm u_{n, j-1})\rangle\\
     = & \sum_{i = 1}^s \langle \bm u_{n, i} - \bm u_{n, i-1}, (\tau\lap_N)^{-1} (\Delta_{i,i} - \frac{1}{2} \sum_{j = 1}^{i-1} \Delta_{i,j} - \frac{1}{2} \sum_{j = i+1}^{s} \Delta_{j, i} )  (\bm u_{n, i} - \bm u_{n, i-1})\rangle \\
   = & \sum_{i = 1}^s \langle \bm u_{n, i} - \bm u_{n, i-1}, (\tau\lap_N)^{-1} (\Delta_{i,i} - \frac{1}{2} \sum_{j = 1}^{i-1} \Delta_{i,j} - \frac{1}{2} \sum_{j = i+1}^{s} \Delta_{j, i} )  (\bm u_{n, i} - \bm u_{n, i-1})\rangle \\
  &  + \sum_{i = 1}^s \sum_{j = 1}^{i-1} \langle \bm u_{n, i} - \bm u_{n, i-1} + \bm u_{n, j} - \bm u_{n, j-1}, \frac{1}{2} (\tau\lap_N)^{-1} \Delta_{i, j} ( \bm u_{n, i} - \bm u_{n, i-1} + \bm u_{n, j} - \bm u_{n, j-1})\rangle.
\end{align*}
Since the matrices in Table \ref{tab:escerk} are assumed to be positive semi-definite, and considering that $\lap_N^{-1}$ is negative definite and commutes with $\Delta_{i,j}$, it is evident that $E(\bm u_{n, s}) - E(\bm u_{n, 0}) \leq 0$.
\end{proof}

The aforementioned procedures demonstrate the crucial role of preserving equilibria, which is achieved using Taylor-polynomial approximations to decompose the energy difference into a combination of quadratic forms. This process enables us to utilize the negativity of quadratic forms to guarantee unconditional energy stability. In the following analysis, we will examine the EFRK schemes on a case-by-case basis and demonstrate that the first- to third-order schemes based on \eqref{eqn:rk123} satisfy the energy stability conditions.

\begin{exm}[EFRK($1, 1$)] We reformulate $N_\kappa(\bm u_{n, 0})$ as
\begin{equation*}
N_\kappa(\bm u_{n, 0}) = (a_{1, 0} \phi_0\tau)^{-1} (\phi_1 \bm u_{n, 1} - \bm u_{n, 0}). \\
\end{equation*}
Then, we calculate
\begin{equation*}
  L \bm u_{n, 1} + N(\bm u_{n, 0}) = (a_{1, 0} \phi_0 \tau)^{-1}(\bm u_{n, 1} - \bm u_{n, 0}) = \tau^{-1}\Delta_{1, 1} (\bm u_{n, 1} - \bm u_{n, 0}).
\end{equation*}
where $\Delta_{1,1} = (a_{1,0}\phi_0(-\tau L_\kappa))^{-1} $. Note that $\phi_0(-\tau L_\kappa) = I$, and $a_{1, 0} = 1$ for the underlying RK(1, 1) parameters; as a result, we get $\Delta_{1, 1} = I$ for EFRK(1,1) scheme, which is positive definite. Therefore, it concludes that the EFRK(1, 1) scheme is unconditionally energy stable.
\end{exm}

\begin{exm}[EFRK($2, 2$)] We reformulate $N_\kappa(\bm u_{n, 1})$ as
\begin{equation*}
N_\kappa(\bm u_{n, 1})  = (a_{2, 1} \phi_1\tau)^{-1}\big( \phi_2 \bm u_{n, 2} - \bm u_{n, 0} - \tau a_{2, 0} \phi_0 N_\kappa(\bm u_{n, 0}) \big).
\end{equation*}
Then, we calculate
\begin{align*}
L_\kappa \bm u_{n, 2} + N_\kappa(\bm u_{n, 1}) & = L_\kappa \bm u_{n, 2} + ( a_{2, 1} \phi_1 \tau)^{-1}\big( \phi_2 \bm u_{n, 2} - \bm u_{n, 0} -  a_{2, 0} \phi_0 \tau N_\kappa(u^n) \big) \\
& = \big(L_\kappa + ( a_{2, 1} \phi_1 \tau)^{-1} \phi_2\big) \bm u_{n, 2} - ( a_{2,1} \phi_1 \tau)^{-1} \bm u_{n, 0} \\
& \quad  - ( a_{2,1} \phi_1 \tau)^{-1}  a_{2, 0} \phi_0 \tau ( a_{1, 0} \phi_0 \tau)^{-1} ( \phi_1 \bm u_{n, 1} - \bm u_{n, 0}) \\
& = (a_{2, 1} \phi_1 \tau)^{-1} (I -  a_{2, 0} \phi_0 \tau L_\kappa)  (\bm u_{n, 2} - \bm u_{n, 1}) + (a_{2, 1} a_{1, 0} \phi_1 \tau)^{-1} (a_{1, 0} - a_{2, 0})  (\bm u_{n, 1} - \bm u_{n, 0}) \\
& = \tau^{-1}\Delta_{2, 2}(\bm u_{n, 2} - \bm u_{n, 1}) + \tau^{-1}\Delta_{2, 1} (\bm u_{n, 1} - \bm u_{n, 0}),
\end{align*}
where $\Delta_{2,1} = (a_{2,1}a_{1,0} \phi_1)^{-1} (a_{1,0} - a_{2,0})$ and $\Delta_{2,2} = (a_{2, 1} \phi_1 )^{-1} (I -  a_{2, 0} \phi_0 \tau L_\kappa) $. The corresponding energy stability matrices are presented as follows:
\begin{align*}
  \Delta_{2,1} & = (a_{2,1}a_{1,0} \phi_1)^{-1} (a_{1,0} - a_{2,0}), \\
  \Delta_{1,1} - \frac{1}{2} \Delta_{2,1}  & =  (a_{1,0} u)^{-1} - (2a_{2,1} a_{1,0}\phi_1 )^{-1} (a_{1,0} - a_{2,0})\\
									   &= (2 a_{2,1} a_{1,0} \phi_1)^{-1}(2a_{2,1}\phi_1 - a_{1,0} + a_{2,0})\\
									   &= (2 a_{2,1} a_{1,0} \phi_1 )^{-1}(c_2 + a_{2,1} - a_{1,0} - 2  a_{2,1} a_{1,0}\tau L_\kappa), \\
  \Delta_{2,2} - \frac{1}{2} \Delta_{2,1} & =  (a_{2,1} \phi_1 )^{-1}(I - a_{2,0} \tau L_\kappa) - (2 a_{2,1}a_{1,0} \phi_1 \tau)^{-1} (a_{1,0} - a_{2,0})\\
									   &= (2 a_{2,1} a_{1,0}\phi_1 )^{-1}\left((a_{1,0} + a_{2,0})I  - 2a_{2,0} a_{1,0}\tau L_\kappa\right).
  \end{align*}

Substituting the RK(2,2) parameters \eqref{eqn:rk123} into energy stability matrices, we have 
\begin{align*}
  \Delta_{2,1} = (I - \tau L_\kappa)^{-1},\ \Delta_{1,1} - \frac{1}{2} \Delta_{2,1}  = (I - \tau L_\kappa)^{-1}(\frac{1}{2}I -\tau L_\kappa),\ \Delta_{2,2} - \frac{1}{2} \Delta_{2,1}  = (I - \tau L_\kappa)^{-1}( \frac{3}{2}I  - \tau L_\kappa).
  \end{align*}
Since that $-\tau L_\kappa$ is semi-positive definite, we can determine that the EFRK($2,2$) scheme is unconditionally energy stable.

\end{exm}

\begin{exm}[EFRK(3, 3)] We reformulate $N_\kappa(\bm u_{n,2})$ as
\begin{equation*}
N_\kappa(\bm u_{n, 2}) = (a_{3, 2} \phi_2 \tau)^{-1} \big(\phi_3 \bm u_{n, 3} - \bm u_{n, 0} - \tau a_{3, 0} \phi_0 N_\kappa(\bm u_{n, 0}) - \tau a_{3, 1} \phi_1 N_\kappa(\bm u_{n, 1})\big).
\end{equation*}
Then, we calculate
\begin{align*}
  L_\kappa \bm u_{n,3} + N_\kappa(\bm u_{n, 2}) & = L_\kappa \bm u_{n, 3} + (a_{3, 2} \phi_2 \tau)^{-1} \big(\phi_3 \bm u_{n, 3} - \bm u_{n, 0} - a_{3, 0} \phi_0 \tau N_\kappa(\bm u_{n, 0}) - a_{3, 1} \phi_1 \tau N_\kappa(\bm u_{n, 1})\big) \\
  & = \big(L_\kappa + (a_{3,2} \phi_2 \tau )^{-1} \phi_3\big) \bm u_{n, 3} - (a_{3, 2} \phi_2 \tau )^{-1} \bm u_{n, 0}  \\
  & \quad - ( a_{3,2}\phi_2 \tau)^{-1}  a_{3, 0} \phi_0  \tau (a_{1, 0} \phi_0 \tau)^{-1} (\phi_1 \bm u_{n, 1} - \bm u_{n, 0}) \\
  & \quad - (a_{3,2} \phi_2 \tau)^{-1} a_{3,1} \phi_1 \tau (a_{2, 1} \phi_1 \tau)^{-1}\big( \phi_2 \bm u_{n, 2} - \bm u_{n, 0} - a_{2, 0} \phi_0 \tau (a_{1, 0} \phi_0 \tau)^{-1} ( \phi_1 \bm u_{n, 1} - \bm u_{n, 0}) \big) \\
  & = ( a_{3, 2} \phi_2 \tau)^{-1} (I - a_{3, 0} \phi_0 \tau L_\kappa - a_{3,1} \phi_1 \tau L_\kappa) \bm u_{n,3} - (a_{3, 2} \phi_2 \tau)^{-1} \bm u_{n, 0}  \\
  & \quad - (a_{3,2} a_{1, 0} \phi_2 \tau)^{-1} a_{3, 0} \phi_1 \bm u_{n, 1} + (a_{3,2} a_{1,0} \phi_2 \tau)^{-1} a_{3,0}  \bm u_{n, 0} \\
  & \quad - (a_{3,2} a_{2,1} \phi_2 \tau)^{-1} a_{3,1} \phi_2  \bm u_{n,2} + (a_{3,2} a_{2,1} \phi_2 \tau)^{-1} a_{3,1}  \bm u_{n, 0} \\
  & \quad + (a_{3,2} a_{2,1} a_{1,0} \phi_2 \tau)^{-1} a_{3,1} a_{2,0}  \phi_1  \bm u_{n,1}
  -  (a_{3,2} a_{2,1} a_{1,0} \phi_2 \tau)^{-1} a_{3,1} a_{2,0}  \bm u_{n, 0} \\
  & = (a_{3, 2} \phi_2 \tau)^{-1} (I - a_{3, 0} \phi_0 \tau L_\kappa - a_{3,1} \phi_1 \tau L_\kappa) (\bm u_{n,3} - \bm u_{n, 2})  \\
  & \quad + (a_{3,2} a_{2,1} \phi_2 \tau)^{-1} \big((a_{2,1} - a_{3,1})I - (a_{3,0} a_{2,1} - a_{3,1}a_{2,0}) \phi_0 \tau L_\kappa\big) (\bm u_{n,2} - \bm u_{n,1}) \\
  & \quad + (a_{3,2} a_{2,1} a_{1, 0} \phi_2 \tau)^{-1} (a_{3,1} a_{2,0}  + a_{2,1} a_{1,0} - a_{3,1}a_{1,0}  - a_{3,0} a_{2,1})   (\bm u_{n,1} - \bm u_{n, 0}) \\
  &= \tau^{-1}\Delta_{3, 3} (\bm u_{n, 3} - \bm u_{n, 2}) + \tau^{-1}\Delta_{3, 2} (\bm u_{n, 2} - \bm u_{n, 1}) + \tau^{-1}\Delta_{3, 1}(\bm u_{n, 1} - \bm u_{n, 0}),
\end{align*}
where $\Delta_{3,1} = (a_{3,2} a_{2,1} a_{1, 0} \phi_2)^{-1} (a_{3,1} a_{2,0}  + a_{2,1} a_{1,0} - a_{3,1}a_{1,0}  - a_{3,0} a_{2,1})$, $\Delta_{3,2} = (a_{3,2} a_{2,1} \phi_2)^{-1} \big((a_{2,1} - a_{3,1})I - (a_{3,0} a_{2,1} - a_{3,1}a_{2,0}) \phi_0 \tau L_\kappa\big)$ and $\Delta_{3,3} = (a_{3, 2} \phi_2)^{-1} (I - a_{3, 0} \phi_0 \tau L_\kappa - a_{3,1} \phi_1 \tau L_\kappa)$. The corresponding energy stability matrices are presented as follows:
\begin{equation*}
\begin{aligned}
  \Delta_{2,1} & = (a_{2, 1} a_{1, 0} \phi_1 )^{-1} (a_{1, 0} - a_{2, 0}), \\
  \Delta_{3,1} & = (a_{3,2} a_{2,1} a_{1, 0} \phi_2 )^{-1} (a_{3,1} a_{2,0} + a_{2,1} a_{1,0} - a_{3,1}a_{1,0}  - a_{3,0} a_{2,1}), \\
  \Delta_{3,2} & = (a_{3,2} a_{2,1} \phi_2 )^{-1} \big((a_{2,1} - a_{3,1})I - (a_{3,0} a_{2,1} - a_{3,1}a_{2,0}) \tau L_\kappa\big), \\
  \end{aligned}
  \end{equation*}
  \begin{equation*}
  \begin{aligned}
  \Delta_{1,1} - \frac{1}{2}\Delta_{2,1} - \frac{1}{2} \Delta_{3,1} & = (2 a_{3,2} a_{2,1} a_{1,0} \phi_2 \phi_1 )^{-1} \big(2 a_{3, 2} a_{2, 1} \phi_2 \phi_1 -  (a_{1, 0} - a_{2, 0}) a_{3, 2} \phi_2 \\
  & \quad -  (a_{3,1} a_{2,0} + a_{2,1} a_{1,0} - a_{3,1}a_{1,0}  - a_{3,0} a_{2,1}) \phi_1 \big), \\
  \Delta_{2,2} - \frac{1}{2} \Delta_{2,1} - \frac{1}{2} \Delta_{3,2} & = (2 a_{3,2} a_{2,1} a_{1,0} \phi_2 \phi_1 )^{-1}\big(a_{3,2} \phi_2[ (a_{1,0} + a_{2,0})I - 2 a_{2,0} a_{1,0} \tau L_\kappa] \\
  & \quad - a_{1,0} \phi_1 [(a_{2,1} - a_{3,1})I - (a_{3,0} a_{2,1} - a_{3,1} a_{2,0}) \tau L_\kappa] \big), \\
  \Delta_{3,3} - \frac{1}{2} \Delta_{3,1} - \frac{1}{2} \Delta_{3,2} & = (2 a_{3,2} a_{2,1} a_{1,0} \phi_2 )^{-1}\big( 2 a_{2,1} a_{1,0}[I - a_{3,0} \tau L_\kappa - a_{3,1} \phi_1 \tau L_\kappa]  \\
  & \quad -  [a_{3,1} a_{2,0} - a_{3,1} a_{1,0} - a_{3,0} a_{2,1} + a_{2,1} a_{1,0}]I  \\
  & \quad -  a_{1,0}[ (a_{2,1} - a_{3,1})I - (a_{3,0} a_{2,1} - a_{3,1} a_{2,0}) \tau L_\kappa]\big).
\end{aligned}
\end{equation*}

Substituting the coefficients of the RK(3, 3) \eqref{eqn:rk123} into the above energy stability matrices, we get
\begin{align*}
  \Delta_{2,1} & = (\frac{2}{3}I - \frac{2}{9}\tau L_\kappa )^{-1},\ \Delta_{3,1} = (3I - 2\tau L_\kappa + \frac{2}{3}(\tau L_\kappa)^2)^{-1},\ \Delta_{3,2} = (3I - 2\tau L_\kappa + \frac{2}{3}(\tau L_\kappa)^2)^{-1}(4I - \tau L_\kappa), \\
  \Delta_{1,1} &- \frac{1}{2}\Delta_{2,1} - \frac{1}{2} \Delta_{3,1} =  (9I - 9\tau L_\kappa + 4(\tau L_\kappa)^2 - \frac{2}{3}(\tau L_\kappa)^3)^{-1}(\frac{75}{4}I - 11\tau L_\kappa + \frac{21}{2}(\tau L_\kappa)^2 - 2(\tau L_\kappa)^3), \\
  \Delta_{2,2} &- \frac{1}{2} \Delta_{2,1} - \frac{1}{2} \Delta_{3,2} = (9I - 9\tau L_\kappa + 4(\tau L_\kappa)^2 - \frac{2}{3}(\tau L_\kappa)^3)^{-1} (\frac{3}{4}I - \tau L_\kappa + (\tau L_\kappa)^2), \\
  \Delta_{3,3} &- \frac{1}{2} \Delta_{3,1} - \frac{1}{2} \Delta_{3,2} = (3I - 2\tau L_\kappa + \frac{2}{3}(\tau L_\kappa)^2)^{-1}(\frac{3}{2}I - \frac{1}{2}\tau L_\kappa).
\end{align*}
Given that $\tau L_\kappa$ is semi-negative definite, it follows that $-\tau L_\kappa$, $(\tau L_\kappa)^2$ and $-(\tau L_\kappa)^3$ are semi-positive definite. Thus, we can confirm the EFRK(3, 3) scheme is unconditionally energy stable.
\end{exm}

\begin{rem}
Recently, H. Liao et al. \cite{liao2024average} and Z. Fu et al. \cite{fu2024higher} independently proposed their work to establish the energy stability of the third-order exponential Runge--Kutta (ERK) schemes \cite{hochbruck2010exponential}. Although the scheme we focus on and the techniques we use are different, the goal is the same: all the work yields the following estimate for the energy difference:
\begin{equation*}
    E(u^{n+1}) -E(u^n) \leq \sum_{i=1}^{s}\sum_{j=1}^{i} \left\langle u_{n,i} - u_{n,i-1} ,(\tau\lap_N)^{-1} \Delta_{i,j} (u_{n,j} - u_{n,j-1}) \right\rangle.
\end{equation*}
However, we use different manners to tackle the right-hand side. Liao and Fu et al. verify the energy dissipation by examining the positive definiteness of operator matrix $\{\Delta_{i,j}\}_{1\leq j\leq i, 1 \leq i\leq s}$, which is succinct and straightforward. Conversely, by arranging the summation of the inner product on the right-hand side, we derive the energy stability matrices present in Table \ref{tab:escerk} and establish the energy stability by evaluating the positive definiteness of the energy stability matrices. From the results point of view, both strategies are capable of establishing the energy stability of the EFRK scheme. They provide a concise validation method, whereas we provide an alternative proof strategy that enriches the analytical method for energy stability.
\end{rem}

\begin{rem}
If we consider the CH equation with variable mobility, this introduces challenges for dealing with the stiff term in the EFRK framework. To the best of our knowledge, while many studies \cite{kim2007numerical,zhu1999coarsening,chen2022preconditioned} have developed effective numerical schemes for solving the CH equation with variable mobility, there is a scarcity of research that investigates the original energy stability from a theoretical standpoint. Within the EFRK framework, energy stability could be established for the first-order scheme by fixing the mobility function using the solution at the previous step. However, extending this stability to second and higher-order algorithms is highly challenging and may require techniques beyond the current scope. This is an important topic for future research to broaden the applicability and enrich the theoretical result of numerical methods in such scenarios.
\end{rem}

\begin{rem}
\label{remark:modif_potential}
    The assumption in Lemma \ref{lem:energy_decrease} implies that the numerical solution is uniformly bounded in the $\ell^\infty$ norm, which is essential in demonstrating that the stabilization EFRK scheme satisfies the energy dissipation law. An alternative approach to fulfill this assumption is to assume that the nonlinear term $f(\cdot)$ satisfies the global Lipschitz condition \cite{shen2010numerical,feng2013stabilized}. In the case of a cubic nonlinear function, these two assumptions are equivalent and necessary for the analysis of the energy dissipation and convergence of the proposed scheme. However, the Cahn--Hilliard equation does not inherently possess a maximum bound principle, and it has not been proven that the solution obtained by EFRK is bounded in maximum norm. Therefore, a common practice \cite{shen2010numerical} is to modify the potential by truncating it when $|u|>\beta$, where $\beta > 1$ is a given large number. In this case, the potential is replaced with a quadratic function, i.e.,
	\begin{equation*}
		\tilde{F}(u)= 
		\left\{
			\begin{aligned}
				&\frac{3\beta^2-1}{2}u^2 - 2\beta^3u + \frac{1}{4}(3\beta^4 +1), &&u>\beta\\
				&\frac{1}{4}(u^2-1)^2, &&|u|\leq \beta\\
				&\frac{3\beta^2-1}{2}u^2 + 2\beta^3u + \frac{1}{4}(3\beta^4 +1), &&u<-\beta\\
			\end{aligned}.
		\right.
	\end{equation*}
	Consequently, the the nonlinear term $f(u)$ is replaced by
	\begin{equation}
		\tilde{f}(u) = \tilde{F}'(u) = 
		\left\{
			\begin{aligned}
				&(3\beta^2-1)u - 2\beta^3, &&u>\beta\\
				&(u^2-1)u, &&|u|\leq \beta\\
				&(3\beta^2-1)u + 2\beta^3, &&u<-\beta\\
			\end{aligned}.
		\right.
	\end{equation}
As a result, after the reformulation, it holds that $|\tilde f'(u)| \leq (3\beta^2-1)$. On the other hand, since taking a large value of $\kappa$ will introduce a significant error \cite{xu2019stability,zhang2023efficient,lee2019effective}, choosing an appropriate value for $\beta$ is necessary. For most scientific computing problems involving the CH equation, the phase variable is always separated in the regions that almost reach the limit $\pm 1$ values due to the mechanism of double well potential \cite{li2022stability}. Thus, despite the lack of theoretical justification, in practice, the phase variable is always bounded in the maximum norm of order $\mathcal{O}(1)$ and is independent of the viscosity coefficient $\epsilon$. Based on this understanding and observation in numerical tests, we find that taking $\beta = \frac{\sqrt{15}}{3}$ does not affect the numerical simulation. Hence, the stabilization parameter can be effectively taken as $\kappa = 2$ in the subsequent numerical experiments.
\end{rem}

\subsection{Linear stability analysis}
We further investigate the linear stability of EFRK schemes. Following the approach of Cox and Matthews \cite{cox2002exponential}, we consider a 1D linear test equation:
\begin{equation*}
	 u_t = \Delta (-\epsilon^2 \Delta u + \lambda u) = (-\epsilon^2 \Delta^2 u + \theta \lambda \Delta u) + (1-\theta) \lambda \Delta u.
\end{equation*}
Here, we assume that $\text{Re}(\lambda)\geq0$. Let $\kappa = \theta \lambda$ be the stabilization parameter; the stabilization Fourier pseudo-spectral discretization has the form
\begin{equation}
	\label{eqn:sfsd}
	\bm{u}_t = L_\kappa \bm{u} + (1-\theta)\lambda\lap_N \bm{u}.
\end{equation}
Applying an $s$-stage, $p$th-order explicit EFRK scheme  yields
\begin{equation*}
	\bm u_{n,i} = \left(\phi_i(-c_i\tau L_\kappa)\right)^{-1} \Big( \bm u^n + \tau \sum_{j=0}^{i-1} a_{i,j} \phi_j(-c_j\tau L_\kappa) (1 -\theta) \lambda \lap_N \bm u\Big).
\end{equation*}

\begin{lem}
	The solution of an $s$-stage EFRK scheme to \eqref{eqn:sfsd} in the Fourier frequency space is
	\begin{equation*}
		\hat{u}_k^{n+1} = \frac{\phi_s(-\tau(1-\theta)\lambda |k|^2)}{\phi_s(\tau(\epsilon^2|k|^4+\theta\lambda |k|^2))}\hat{u}_k^n, \quad k = -N_1/2, \dots, N_1/2.
	\end{equation*}
\end{lem}
\begin{proof}
	Let $\hat{u}_{n,i,k}$ be the $k$th Fourier mode of $\bm u_{n,i}$. Then, the first-stage solution in the Fourier frequency space satisfies
	\begin{equation}
		\begin{aligned}
			\phi_1(c_1\tau(\epsilon^2k^4+\theta\lambda k^2))\hat{u}_{n,1,k} &= \hat{u}_{n,0,k} - \tau a_{1, 0}  (1-\theta)\lambda k^2 \hat{u}_{n,0,k}\\
										                                    &= \phi_1(-c_1\tau(1-\theta)\lambda k^2)\hat{u}_{n,0,k}.
		\end{aligned}
	\end{equation}
	Assuming that $\hat{u}_{n,j,k} = \frac{\phi_j(-c_j\tau(1-\theta)\lambda k^2)}{\phi_j(c_j\tau(\epsilon^2k^4+\theta\lambda k^2))}\hat{u}_{n,0,k}$, $j = 0,\dots,i-1$, we obtain
	\begin{equation*}
		\begin{aligned}
			\hat{u}^n_{i,k} &= \frac{1}{\phi_i(c_i\tau(\epsilon^2k^4+\theta\lambda k^2))} \big(\hat{u}_{n,0,k} - \tau  \sum_{j=0}^{i-1}a_{i,j} \phi_j(c_j\tau(\epsilon^2k^4+\theta\lambda k^2)) (1-\theta)\lambda k^2  \hat{u}_{n,j,k}\big)\\
						 	&= \frac{1}{\phi_i(c_i\tau(\epsilon^2k^4+\theta\lambda k^2))}\big( 1- \tau \sum_{j=0}^{i-1} a_{i,j} \phi_j(c_j\tau(\epsilon^2k^4+\theta\lambda k^2)) (1-\theta)\lambda k^2  \frac{\phi_j(-c_j\tau(1-\theta)\lambda k^2)}{\phi_j(c_j\tau(\epsilon^2k^4+\theta\lambda k^2))}\big) \hat{u}_{n,0,k}\\
							&= \frac{\phi_i(-c_i\tau(1-\theta)\lambda k^2)}{\phi_i(c_i\tau(\epsilon^2k^4+\theta\lambda k^2))}\hat{u}_{n,0,k}, \quad i \leq s.
		\end{aligned}
	\end{equation*}
	Finally, we get
	\begin{equation*}
		\hat{u}^{n+1}_{k}	= \frac{\phi_s(-\tau(1-\theta)\lambda k^2)}{\phi_s(\tau(\epsilon^2k^4+\theta\lambda k^2))}\hat{u}^n_{k}.
	\end{equation*}
 This completes the proof.
\end{proof}
Because the approximations $\phi_s(\cdot)$ have non-negative coefficients and $\text{Re}(\lambda) \geq 0 $, we have
\begin{equation*}
\frac{\phi_s(-\tau(1-\theta)\lambda k^2)}{\phi_s(\tau(\epsilon^2k^4+\theta\lambda k^2))}\leq  \frac{\phi_s(-\tau(1-\theta)\lambda k^2)}{\phi_s(\tau\theta\lambda k^2)}.
\end{equation*}
Denote $z = \tau\lambda k^2$ and recall the definition of the polynomial $\phi_s(\cdot)$,
we define the stability function as
\begin{equation*}
	\Phi(\theta,z) = \frac{\phi_s(-(1-\theta) z)}{\phi_s(\theta z)} = \frac{\sum_{k=0}^{s}\frac{( -(1-\theta) z)^k}{k!}}{\sum_{k=0}^{s}\frac{(\theta z)^k}{k!}},\quad p = s\leq 3.
\end{equation*}

Recalling the definition of stability for a method, $A$-stability is achieved when $\Phi(\theta,z)$ satisfies 
\begin{equation*}
	|\Phi(\theta,z)|\leq 1\quad \text{for all}\quad z\in \mathbb{C}^-,
\end{equation*}
where $\mathbb{C}^-$ denote the left-half complex plane. Through careful calculation, we conclude that the EFRK schemes are A-stable when $\theta = \frac{1}{2}$. Specifically, the stability regions of the EFRK(1,1) and EFRK(2,2) schemes contain the left-half complex plane when $\theta \geq \frac{1}{2}$. For the EFRK(3,3) scheme, only at $\theta = \frac{1}{2}$ is A-stable and the stability regions will exclude part of the left-half complex plane when $\theta > \frac{1}{2}$. We depict the stability boundary curves of EFRK schemes in Fig. \ref*{fig:linrstab}. Clearly, the stability region becomes large when $\theta$ increases to $\frac{1}{2}$. When $\theta = 0$, the stability region of EFRK reduces to that of the classic RK scheme. Moreover, when $\theta = \frac{1}{2}$, the stability region contains the entire left-half complex plane, see Figs. \ref*{fig:linrstab} (a)--(c).
\begin{figure}[ht]
	\centering
	\includegraphics[width=0.95\linewidth]{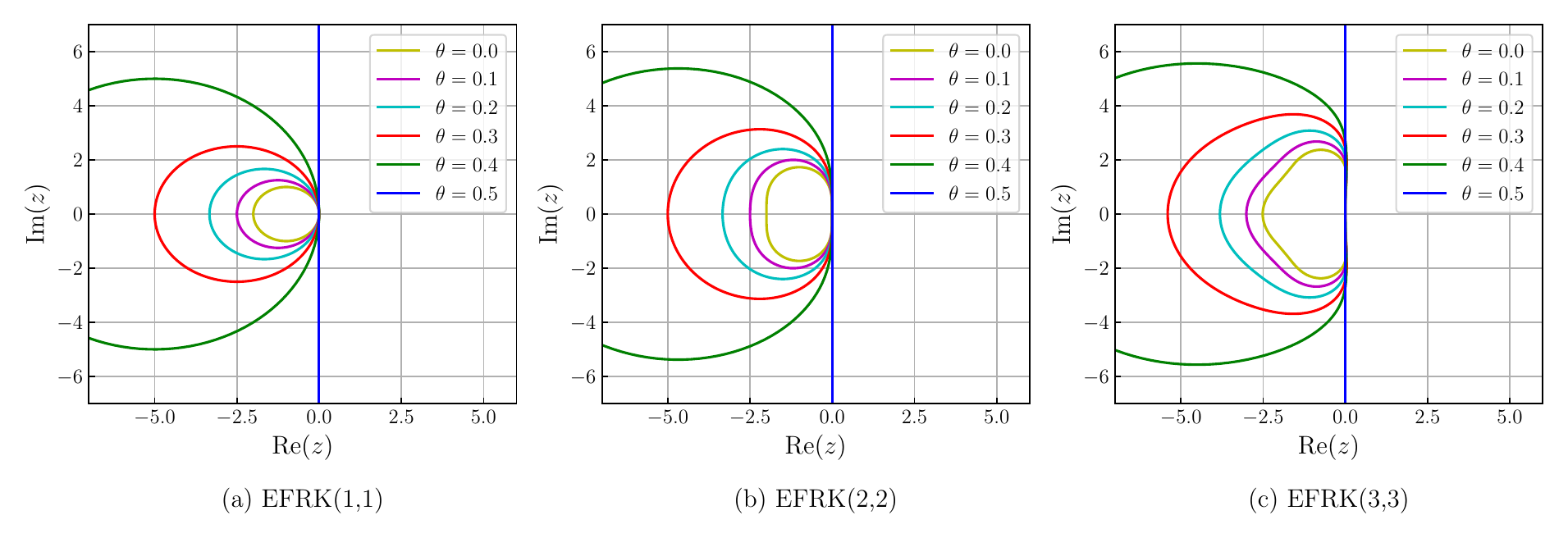}
	\caption{Boundaries of the stability regions for EFRK(s,p) with different values of $\theta$.}
	\label{fig:linrstab}
\end{figure}

\section{Error estimate}
\label{sec:5}

We now proceed to examine the error estimates for the solution derived from the fully discrete scheme in this section. To start, we introduce the following lemma, which is essential for the error analysis.

\begin{lem}[p. 36 of \cite{shen2011spectral}]
\label{lem:estimFPS}
 For any $v \in H^{m}_{per}(\Omega)$ with $m > 1/2$, the following estimate holds for all $0\leq l\leq m$:
 \begin{equation*}
  || \partial_x^l (I_N v - v)||_{L^2} < C_* h^{m-l}.
 \end{equation*}
\end{lem}

First, we estimate the error between the exact solution and the solution of the spatial semi-discrete problem, i.e.,
\begin{equation}\label{eqn:semiLNapdx}
  \frac{\mathrm{d} \bm u}{\mathrm{d} t} = \lap_N (-\epsilon^2 \lap_N \bm u + f(\bm u)).
\end{equation}
Denote $I_h$ as the operator limiting a continuous function on the mesh $\Omega_N$; then, we have the following theorem.

\begin{thm}
    \label{thm:semi_estimate}
  Assume that $u \in H^1(0,T;H^{m+d+4}_{per}(\Omega))$. For any fixed $t\in (0,T]$, we have
  \begin{equation*}
    || I_h u(t) - \bm u(t)||_{\ell^2} \leq C_0 h^{m},
  \end{equation*}
  where $C_0 > 0$ is a constant independent of $h$.
\end{thm}

\begin{proof}
  The exact solution limited on mesh $\Omega_h$ can be considered satisfying \eqref{eqn:semiLNapdx} with a defect $\delta(t)$
  \begin{equation}
    \label{eqn:exactRes}
    \frac{\mathrm{d} I_h u}{\mathrm{d} t} = - \epsilon^2 \lap_N^{2} I_hu + \lap_N f(I_hu) + \delta(t),
  \end{equation}
  where 
  \begin{equation*}
    \begin{aligned}
    \delta(t) &= \epsilon^2\lap_N^{2} I_h u - \epsilon^2 I_h \Delta^2 u + I_h \Delta f(u) - \lap_N f(I_hu).\\
    \end{aligned}
  \end{equation*}
 Under the assumption of the theorem, using Lemma \ref{lem:estimFPS} and Sobolev embedding theorem (Theorem 5.9 of \cite{adams2003sobolev}), we can derive
  \begin{equation}
    \max_{t\in (0,T]}||\delta(t)||_{\ell^2} \leq C_*h^{m}.
  \end{equation}
  Let $\eta(t) = I_hu(t) -  \bm u(t)$, $\forall t\in[0,T]$; then, subtracting \eqref{eqn:semiLNapdx} from \eqref{eqn:exactRes} yields
  \begin{equation}
    \frac{\mathrm{d}\eta}{\mathrm{d}t} = -\epsilon^2 \lap_N^2 \eta + \lap_N (f(I_hu) - f(\bm u)) + \delta(t).
  \end{equation}
      Taking the discrete $\ell^2$ inner product with $2\eta$ and using the modification mentioned in Remark \ref{remark:modif_potential} yields
  \begin{align*}
    \frac{\mathrm{d}}{\mathrm{d}t}||\eta||_{\ell^2}^2 &= -2\epsilon^2||\lap_N \eta||_{\ell^2}^2 + 2\langle \lap_N f(I_hu) - f(\bm u),\eta\rangle + 2\langle \delta, \eta \rangle\\
                        &\leq -2\epsilon^2||\lap_N \eta||_{\ell^2}^2 + 2||f(I_hu) - f(\bm u )||_{\ell^2} ||\lap_N \eta||_{\ell^2} + 2||\delta||_{\ell^2} ||\eta||_{\ell^2}\\
                        &\leq -2\epsilon^2||\lap_N \eta||_{\ell^2}^2 + 2(3\beta^2-1)||\eta||_{\ell^2} ||\lap_N\eta||_{\ell^2} + 2||\delta||_{\ell^2} ||\eta||_{\ell^2}\\
                        &\leq -2\epsilon^2||\lap_N\eta||_{\ell^2}^2 + \frac{(3\beta^2-1)^2}{2\epsilon^2}||\eta||^2_{\ell^2} + 2\epsilon^2 ||\lap_N\eta||^2_{\ell^2} + 2||\delta||_{\ell^2} ||\eta||_{\ell^2}\\
                        &\leq \frac{(3\beta^2-1)^2}{2\epsilon^2}||\eta||_{\ell^2}^2 + ||\delta||_{\ell^2}^2 +  ||\eta||_{\ell^2}^2\\
                        &= \left(\frac{(3\beta^2-1)^2}{2\epsilon^2} + 1\right)||\eta||_{\ell^2}^2 + ||\delta||_{\ell^2}^2.
  \end{align*}
  Letting $C_1 = \frac{(3\beta^2-1)^2}{2\epsilon^2} + 1$ and applying Gr\"onwall's inequality yield
  \begin{equation}
  ||\eta(t)||_{\ell^2}^2 \leq \int_0^t e^{C_1(t-s)}||\delta(s)||_{\ell^2}^2ds < e^{C_1T}\frac{C_*^2}{C_1}h^{2m},
  \end{equation}
  which completes the theorem with $C_0 = \sqrt{e^{C_1T}\frac{C_*^2}{C_1}}$.
\end{proof}

Next, we estimate the error between the semi-discrete solution $\bm u(t)$ obtained by \eqref{eqn:semiLNapdx} and the approximate solution $\bm u^n$ computed by the EFRK scheme.
\begin{thm}
    \label{thm:full_estimate}
    Assume that the exact solution $u \in H^p(0,T;H_{pre}^{m+d+4}(\Omega))$, $\bm u(t)$ is the solution of the semi-discrete solution, and $\bm u^n$ is the approximate solution calculate by the $p$th-order EFRK scheme; then, it holds that
    \begin{equation}
        || \bm u(t) - \bm u^n ||_{\ell^2} \leq C_2 \tau^p.
    \end{equation}
\end{thm}
\begin{proof}

Let us introduce reference solution $U_{n,i}$ that satisfies
\begin{align}
    U_{n,0} &= \bm u(t_n), \notag\\
    U_{n,i} &= \left(\phi_i(-c_i\tau L_\kappa)\right)^{-1} (U_{n,0} + \tau \sum_{j=0}^{i-1}a_{i,j}\phi_j(-c_j\tau L_\kappa)N_\kappa(U_{n,j})),\quad i=1,\dots,s-1, \label{eqn:ref_efrk_i} \\
    U_{n,s} &= \left(\phi_s(-\tau L_\kappa)   \right)^{-1} (U_{n,0} + \tau \sum_{j=0}^{s-1}a_{i,j}\phi_j(-c_j\tau L_\kappa)N_\kappa(U_{n,j}) + \tau R^{n}), \label{eqn:ref_efrk_s}\\
    \bm u(t_{n+1}) &= U_{n,s}, \notag
\end{align}
where $R^n$ is the truncation error. Let us denote $e^n = \bm u(t_n) - \bm u^n$ and $e_{n,i} = U_{n,i} - \bm u_{n,i}$ for $i = 0,1, \cdots, s-1$.
For each $i = 0,1, \cdots, s-1$, the difference between \eqref{eqn:efrk} and \eqref{eqn:ref_efrk_i} is
\begin{equation}
\label{eqn:diff_efrk_i}
e_{n,i} = \left(\phi_{i}(-c_i\tau L_\kappa)\right)^{-1}\Big(e^{n} + \tau \sum_{j=0}^{i-1}a_{i,j}\phi_j(-c_j\tau L_\kappa)(N_\kappa(U_{n,j}) - N_\kappa(u_{n,j}))\Big),\quad i=1,\dots,s-1.
\end{equation}
Acting $1-c_i\tau L_\kappa$ on both sides of \eqref{eqn:diff_efrk_i} and taking the discrete $\ell^2$ inner product with $e_{n,i}$, we have 
\begin{equation}
	\begin{aligned}
	&||e_{n,i}||_{\ell^2}^2 + c_i\tau\epsilon^2||\lap_N e_{n,i}||_{\ell^2}^2 - c_i\tau\kappa\langle \lap_N e_{n,i},e_{n,i}\rangle\\
	&\leq \langle e^n,e_{n,i}\rangle + \tau C_3 \sum_{j=0}^{i-1}a_{i,j}\langle N_\kappa(U_{n,j}) - N_\kappa(u_{n,j}),e_{n,i}\rangle \\
	&\leq \frac{1}{2}||e^n||_{\ell^2}^2 + \frac{1}{2}||e_{n,i}||_{\ell^2}^2 +  \tau\sum_{j=0}^{i-1}a_{i,j}\frac{C_3^2(3\beta^2 - 1)^2}{4\epsilon^2}||e_{n,j}||_{\ell^2}^2 + c_i\tau\epsilon^2||\lap_N e_{n,i}||_{\ell^2}^2,
	\end{aligned}
\end{equation}
where $C_3 := \max\limits_{z\geq 0,i>j} \frac{(1 + c_i z)\phi_{j}(c_j z)}{\phi_{i}(c_i z)}$. Denoting $C_4 := \frac{C_3^2(3\beta^2-1)^2}{2\epsilon^2}$, we obtain
\begin{equation}
	||e_{n,i}||_{\ell^2}^2 \leq ||e^n||_{\ell^2}^2 + C_4\tau \sum_{j=0}^{i-1}a_{i,j}||e_{n,j}||_{\ell^2}^2.
\end{equation}
By induction, assuming that $||e_{n,j}||^2\leq (1 + C_4\tau)^j||e^n||^2$ for $j = 0,1,\dots,i-1$, we obtain
\begin{equation}
	||e_{n,i}||_{\ell^2}^2 \leq ||e^n||_{\ell^2}^2 + C_4\tau \sum_{j=0}^{i-1}(1 + C_4\tau)^j||e^n||_{\ell^2}^2 = (1 + C_4\tau)^i||e^n||_{\ell^2}^2.
\end{equation}
Therefore, the inequality holds for $i = 0,1,\cdots,s-1$.

Similarly, subtracting \eqref{eqn:efrk} from \eqref{eqn:ref_efrk_s} yields
\begin{equation}
	e^{n+1} = \left(\phi_{s}(-\tau L_\kappa)\right)^{-1}\Big( e^{n} + \tau \sum_{j=0}^{s-1}a_{s,j}\phi_j(-c_j\tau L_\kappa)(N_\kappa(U_{n,j})-N_\kappa(u_{n,j})) +\tau R^{n} \Big).
\end{equation}
Applying $I - \tau L_\kappa$ on both sides and taking the $\ell^2$ inner product with $e^{n+1}$ yield
\begin{equation}
	\begin{aligned}
	&\quad\ ||e^{n+1}||_{\ell^2}^2 + \tau\epsilon^2||\lap_N e^{n+1}||_{\ell^2}^2 - \tau\kappa\langle \lap_N e^{n+1},e^{n+1}\rangle\\
	&\leq \langle e^n,e^{n+1}\rangle + \tau C_3 \sum_{j=0}^{s-1}a_{s,j}\langle N_\kappa(U_{n,j}) - N_\kappa(u_{n,j}),e^{n+1}\rangle + \tau\langle R^n,e^{n+1}\rangle\\
	&\leq \frac{1}{2}||e^n||_{\ell^2}^2 + \frac{1}{2}||e^{n+1}||_{\ell^2}^2 + \frac{C_4}{2}\tau\sum_{j=0}^{s-1}a_{s,j}||e_{n,j}||_{\ell^2}^2 +  \tau\epsilon^2||\lap_N e^{n+1}||_{\ell^2}^2 + \frac{\tau C_3^2}{2}||R^n||_{\ell^2}^2 + \frac{\tau}{2}||e^{n+1}||_{\ell^2}^2.
	\end{aligned}
\end{equation}
When $\tau \leq \frac{1}{2}$, we have
\begin{equation}
	\begin{aligned}
	||e^{n+1}||_{\ell^2}^2 &\leq 2||e^n||_{\ell^2}^2 + 2C_4\tau\sum_{j=0}^{s-1}||e^{n,j}||_{\ell^2}^2 + 2\tau C_3^2 ||R^n||_{\ell^2}^2\\
						  &= 2(1 + C_4\tau)^s||e^n||_{\ell^2}^2 + 2\tau C_3^2 ||R^n||_{\ell^2}^2.
	\end{aligned}
\end{equation}
By recursion, we obtain
\begin{equation}
	\begin{aligned}
	||e^{n+1}||_{\ell^2}^2 &\leq 2(1 + C_4\tau)^s||e^n||_{\ell^2}^2 + 2\tau C_3^2 ||R^n||_{\ell^2}^2\\
				  &= 2^n(1 + C_4\tau)^{ns}||e^0||_{\ell^2}^2 + C_3^2\tau ||R^n||_{\ell^2}^2\sum_{k=0}^{n}2^{k+1}(1 + C_4\tau)^{ks}.
	\end{aligned}
\end{equation}

Since $e^0 = 0$, we deduce
\begin{equation}
    \label{eqn:err_form}
	||e^{n+1}||_{\ell^2} \leq \sqrt{2^{n+1}t_nC_3^2(1 + C_4\tau)^{ns}}||R^n||_{\ell^2} \leq \sqrt{2^{n+1}t_nC_3^2e^{C_4 s t_n}}||R^n||_{\ell^2}.
\end{equation}
According to \cite{zhang2023efficient}, for the EFRK scheme with the order of the RK parameters $p\leq 4$, the truncation error satisfies  
\begin{equation}
    \label{eqn:err_trunc}
    ||R^n||_{\ell^2} \leq C_5\tau^p,\quad 0\leq n\leq N-1.
\end{equation}
where $C_5$ is a constant independent of $\tau$. 

Inserting \eqref{eqn:err_trunc} into \eqref{eqn:err_form} and denoting $C_2 = C_5\sqrt{2^{n+1}t_nC_3^2e^{C_4 s t_n}}$, we have
\begin{equation}
    ||e^{n}||_{\ell^2} \leq C_2\tau^p,\quad 1\leq n\leq N,
\end{equation}
which completes the proof.
\end{proof}

By combining Theorems \ref{thm:semi_estimate} and \ref{thm:full_estimate}, we obtain the following result on the error estimate of the full discrete schemes.
\begin{thm}
Assuming that the exact solution $u \in H^p(0,T;H^{m+d+4}(\Omega))$ with $m\geq 2$ and $\bm u^n$ is the approximate solution calculated by the $p$th-order EFRK scheme, it holds that
\begin{equation}
    ||I_h u(t_n) - \bm u^n||_{\ell^2} \leq C(h^{m} + \tau^p),\quad n=1,\cdots, N,
\end{equation}
where $C$ is a constant independent of $h$ and $\tau$.
\end{thm}

\section{Numerical experiments}
\label{sec:6}
In this section, we will conduct a series of numerical experiments to illustrate the effective, high-order accuracy, and unconditional energy stability of EFRK schemes in simulating the coarsening dynamics of the CH equation. Additionally, we present an adaptive time-stepping method that effectively utilizes the equilibrium-preserving property and unconditional energy stability. In the following experiments, the stabilization parameter is set to $\kappa = 2$ as a default unless otherwise specifically indicated. 

\subsection{Test on the preservation of equilibrium state}
\begin{figure}[ht]
    \centering
    \includegraphics[width=0.9\linewidth]{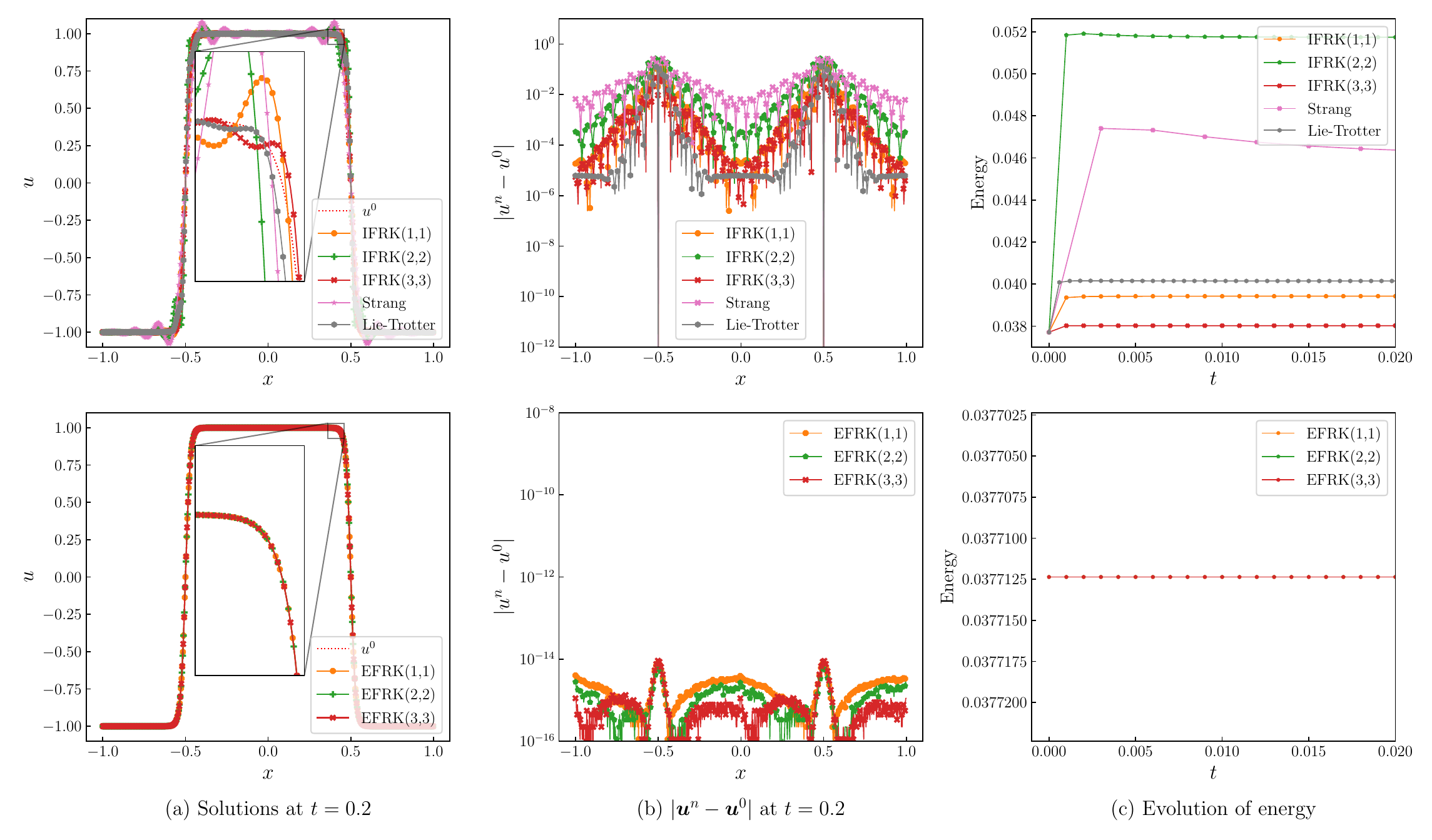}
    \caption{Example \ref*{exm:equilibrium}: the profiles of the solution $u$ (left column), absolute error (middle column), and energy (right column) computed using the IFRK, EFRK, and splitting schemes.}
    \label{fig:equilib}
\end{figure}

\begin{exm}
\label{exm:equilibrium}
To demonstrate the superiority of the proposed EFRK framework, we first examine the 1D Cahn--Hilliard equation with periodic boundary conditions on $\Omega=(-1,1)$. The initial condition is selected as
\begin{equation*}
    {u}_0(x) = \tanh(\frac{0.5-|x|}{\sqrt{2}\epsilon}), 
\end{equation*}
which corresponds to a local minimum of the free energy functional \eqref{eqn:energy}. This initial profile exhibits two equilibrium interfaces at $x = \pm 0.5$.

In addition to the IFRK and EFRK schemes, the popular first-order Lie--Trotter and second-order Strang splitting schemes \cite{li2021energy,li2022operator} are also considered in this example. The splitting method decomposes the vector field of the original equation into several components, allowing the resulting subsystems to be easily integrated. We approximate the solution $\bm{u}$ via the following splitting schemes
\begin{equation*}
    \text{Lie--Trotter:}\  \bm{u}^{n+1} = S_L^{\tau}S_N^{\tau}\bm{u}^n, \quad \text{Strang:}\ \bm{u}^{n+1} = S_L^{\frac{\tau}{2}}S_N^{\tau}S_L^{\frac{\tau}{2}}\bm{u}^n.
\end{equation*}

Here, the operator $S_{L}^{\tau}$ and $S_{N}^{\tau}$ are denoted as the solver of the following linear and nonlinear problem, respectively:
\begin{equation*}
\left\{
    \begin{aligned}
        &\bm{u}_t = L_{\kappa}\bm{u}, \quad 0 < t \leq \tau,\\
        &\bm{u}(0) = \bm{v},
    \end{aligned}
\right.
\qquad
\left\{
    \begin{aligned}
        &\bm{u}_t = N_\kappa(\bm{u}),\quad 0<t\leq \tau,\\
        &\bm{u}(0) = \bm{v}.
    \end{aligned}
\right.
\end{equation*}
It is worth noting the splitting scheme also fails to preserve the equilibrium state. To provide a concise explanation, we approximate the solution using the Lie--Trotter scheme, which can be represented as:
\begin{equation*}
    u^{n+1} = e^{\tau L_\kappa}\left(u^n + \int_{0}^{\tau} N_\kappa(u(t_n + s)) ds \right)
\end{equation*}
Since the integral on the right-hand side generally does not have the primitive function expression, we approximate it using the left rectangle rule and obtain the first-order scheme
\begin{equation}\label{eqn:lietrotter}
    u^{n+1} = e^{\tau L_\kappa}\left(u^n + \tau N_\kappa(u^n) \right)
\end{equation}
Then by substituting $u^n$ in \eqref{eqn:lietrotter} with the equilibrium state $\bm u^*$, we get
\begin{equation*}
    \bm u^{n+1} = e^{\tau L_\kappa}(\bm u^* + \tau N_\kappa(\bm u^*)) = e^{\tau L_\kappa}(I - \tau L_\kappa)\bm u^*.
\end{equation*}
Using the same analysis in subsection 3.1, it is clear that the splitting method fails to maintain steady state and suffers from exponential decay effects.
\end{exm}

For the sake of comparison, we adopt the following parameter settings: $\epsilon^2 = 0.0004$, $\tau = 0.0005$, $N = 2048$, $\kappa = 0$, and final time $T = 0.02$. To ensure computational stability, we specifically chose a time step of $\tau = 0.0001$ for the Lie--Trotter scheme. In addition, because the nonlinear solver $S_{N}^T$ in the splitting scheme cannot be displayed in an explicit expression, we resort to using the RK($1,1$) method for Lie--Trotter scheme and RK($2,2$) method for Strang scheme in computation. All the numerical results are presented in Fig. \ref{fig:equilib}.
	
The top row of Fig. \ref{fig:equilib} illustrates that both IFRK and operator splitting schemes fail to preserve the equilibrium state even in the absence of a stabilization term. While higher-order IFRK schemes improve accuracy, they continue to exhibit significant errors caused by exponential damping near the interfacial regions ($x \approx \pm 0.5$). As the equilibrium state changes, the solutions obtained by the IFRK and operator splitting schemes are no longer the local minimum of the original energy. This leads to the energy rise reflected in the profile, which contradicts the intrinsic energy-decreasing property of the CH equation. The outcome reveals the numerical schemes that do not preserve the steady state are not suitable for solving the Cahn–Hilliard equation and, more generally, any gradient flow equation.

In contrast, by utilizing the Taylor polynomial approximations to the exponential functions, all EFRK schemes shown in the bottom row of Fig. \ref{fig:equilib} effectively maintain the initial equilibrium state. Although relatively higher errors introduced by the spatial discretization are noticeable near the interface layers, they are more accurate than those of IFRK and splitting schemes. In Fig. \ref{fig:equilib}(c), it is evident that the energy profiles remain constant over time. Hence, to preserve the equilibrium state and ensure energy dissipation in gradient flow systems, it is crucial to preserve equilibrium states.

\subsection{Convergence tests}

In this subsection, we will verify the temporal and spatial convergence of the proposed schemes.

\begin{exm}
\label{exm:convtest}
Consider the 1D CH equation \eqref{eqn:chsys} with the initial condition
\begin{equation*}
	u_0(x) = 0.1 \big(\sin(3\pi x) + \sin(5\pi x)\big),
\end{equation*}
on the interval $\Omega = (-1,1)$, the terminal time is set to be $T = 0.1$. 
\end{exm}
\begin{table}[ht]
\begin{tabular}{|c|c|c|c|c|c|c|c|c|c|}
\hline
$\tau$          & Scheme                     & $\ell^2$ Error & Order         & Scheme                     & $\ell^2$ Error & Order & Scheme                     & $\ell^2$ Error & Order \\ \hline
$\delta/2^6$    & \multirow{6}{*}{EFRK(1,1)} & 4.0900e-03  &  -            & \multirow{6}{*}{EFRK(2,2)} & 8.5267e-05 &  -     & \multirow{6}{*}{EFRK(3,3)} & 1.2587e-06 &  -      \\
$\delta/2^7$    &                            & 2.0852e-03  &  0.97         &                            & 2.1814e-05  &  1.97  &                            & 1.6290e-07  &  2.95   \\
$\delta/2^8$    &                            & 1.0529e-03  &  0.99         &                            & 5.5181e-06  &  1.98  &                            & 2.0778e-08  &  2.97   \\
$\delta/2^9$    &                            & 5.2903e-04  &  0.99         &                            & 1.3878e-06  &  1.99  &                            & 2.6260e-09  &  2.98   \\
$\delta/2^{10}$   &                            & 2.6517e-04  &  1.00         &                            & 3.4801e-07  &  2.00  &                            & 3.2917e-10  &  3.00   \\
$\delta/2^{11}$   &                            & 1.3275e-04  &  1.00         &                            & 8.7134e-08  &  2.00  &                            & 3.9785e-11  &  3.05   \\ \hline
\end{tabular}
\caption{Errors and convergences rates of the EFRK schemes with $\epsilon^2=0.01$, $N_1 = 512$, and $\delta = 10^{-2}$.}
\label{tab:convRate01}
\vspace*{0.4cm}
\begin{tabular}{|c|c|c|c|c|c|c|c|c|c|}
\hline
$\tau$          & Scheme                     & $\ell^2$ Error & Order         & Scheme                     & $\ell^2$ Error & Order & Scheme                     & $\ell^2$ Error & Order \\ \hline
$\delta/2^6$    & \multirow{6}{*}{EFRK(1,1)} & 3.9485e-04  &  -            & \multirow{6}{*}{EFRK(2,2)} & 7.8463e-06 &  -     & \multirow{6}{*}{EFRK(3,3)} & 4.2049e-06 &  -      \\
$\delta/2^7$    &                            & 2.0906e-04  &  0.92         &                            & 3.4195e-06  &  1.20  &                            & 8.5690e-07  &  2.29   \\
$\delta/2^8$    &                            & 1.0804e-04  &  0.95         &                            & 1.1666e-06  &  1.55  &                            & 1.4953e-07  &  2.52   \\
$\delta/2^9$    &                            & 5.4982e-05  &  0.97         &                            & 3.4902e-07  &  1.74  &                            & 2.3264e-08  &  2.68   \\
$\delta/2^{10}$   &                            & 2.7741e-05  &  0.99         &                            & 9.6653e-08  &  1.85  &                            & 3.3204e-09  &  2.81   \\
$\delta/2^{11}$   &                            & 1.3935e-05  &  0.99         &                            & 2.5565e-08  &  1.92  &                            & 4.4054e-10  &  2.91   \\ \hline
\end{tabular}
\caption{Errors and convergences rates of the EFRK schemes with $\epsilon^2=0.0025$, $N_1 = 512$, and $\delta = 10^{-2}$.}
\label{tab:convRate0025}
\end{table}

We begin by validating the temporal convergence of the EFRK schemes. In this experiment, we fix the mesh number as $N = 512$ and set $\delta = 10^{-2}$. Given the difficulty in obtaining the explicit analytic solution of CH equation, we utilize the numerical solution obtained by the classical ERK(3,3) \cite[Scheme (5.8) with $c_2 = \frac{4}{9}$]{hochbruck2005explicit}, with $\tau = \delta / 2^{12}$, as the reference solution. We then carry out numerical simulations for the EFRK schemes with various time steps $\tau = \delta/2^k$ for $k = 6, 7, \dots, 11$, and compute the error with respect to $\tau$. Table \ref{tab:convRate01} and Table \ref{tab:convRate0025} list the error in $\ell^2$ norm and convergence rates with $\epsilon^2 = 0.01,0.0025$ for EFRK schemes, respectively. As expected, we clearly observe first- to third-order accuracy in time for the corresponding EFRK schemes.

Subsequently, we evaluate the spectral accuracy of the spectral collocation method by setting a fixed time step of $\tau = \delta/2^{12}$ and employing the EFRK(3,3) scheme for temporal discretization. The reference solution is computed with $N = 2^{11}$. To assess the numerical error at the final time, we uniformly refine the spatial grid from $N = 2^2$ to $N = 2^{10}$. Fig. \ref{fig:spatial_convg} illustrates the spectral accuracy of the Fourier pseudo-spectral discretization by showing the marked decrease in numerical error.
 
\begin{figure}[htbp]
	\centering
	\includegraphics[width=0.6\linewidth]{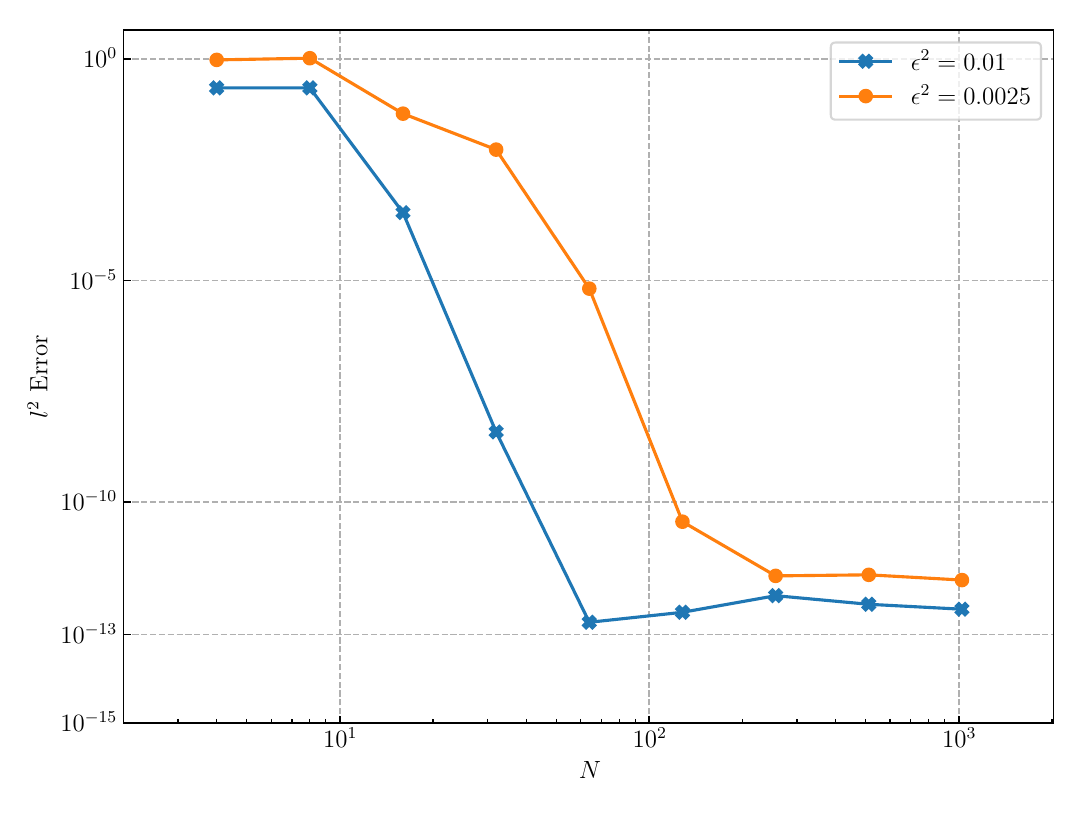}
	\caption{Example \ref*{exm:convtest}: spatial accuracy tests of the Fourier pseudo-spectral method.}
	\label{fig:spatial_convg}
\end{figure}

\begin{exm}
\label{exm:advantage}
As analyzed in \cite{xu2019stability,zhang2023efficient,lee2019effective}, while the stabilization method expands the stability region and allows for the use of large time steps, it also introduces additional errors in the numerical solution. Consequently, ensuring the accuracy of solutions requires the use of small step sizes during the computation. In this example, our objective is to verify the energy dissipation law and mass conservation of the proposed scheme as well as to demonstrate that higher-order schemes can mitigate the impact of the stabilization technique.

We consider the 2D CH equation \eqref{eqn:chsys} with diffusion coefficients $\epsilon^2 = 0.0025$ with terminal time $T = 100$. The computational domain is set to $\Omega = (-\pi,\pi)^2$, and the mesh point numbers are fixed as $N_1 = N_2 = 128$. The initial condition is given by a random value at each point that is uniformly distributed in $[-0.5,0.5]$. The reference solution is obtained using the $ERK(3,3)$ scheme with $\tau = 10^{-5}$ and $\kappa = 0$.

\end{exm}

\begin{figure}[htbp]
	\centering
	\includegraphics[width=0.9\linewidth]{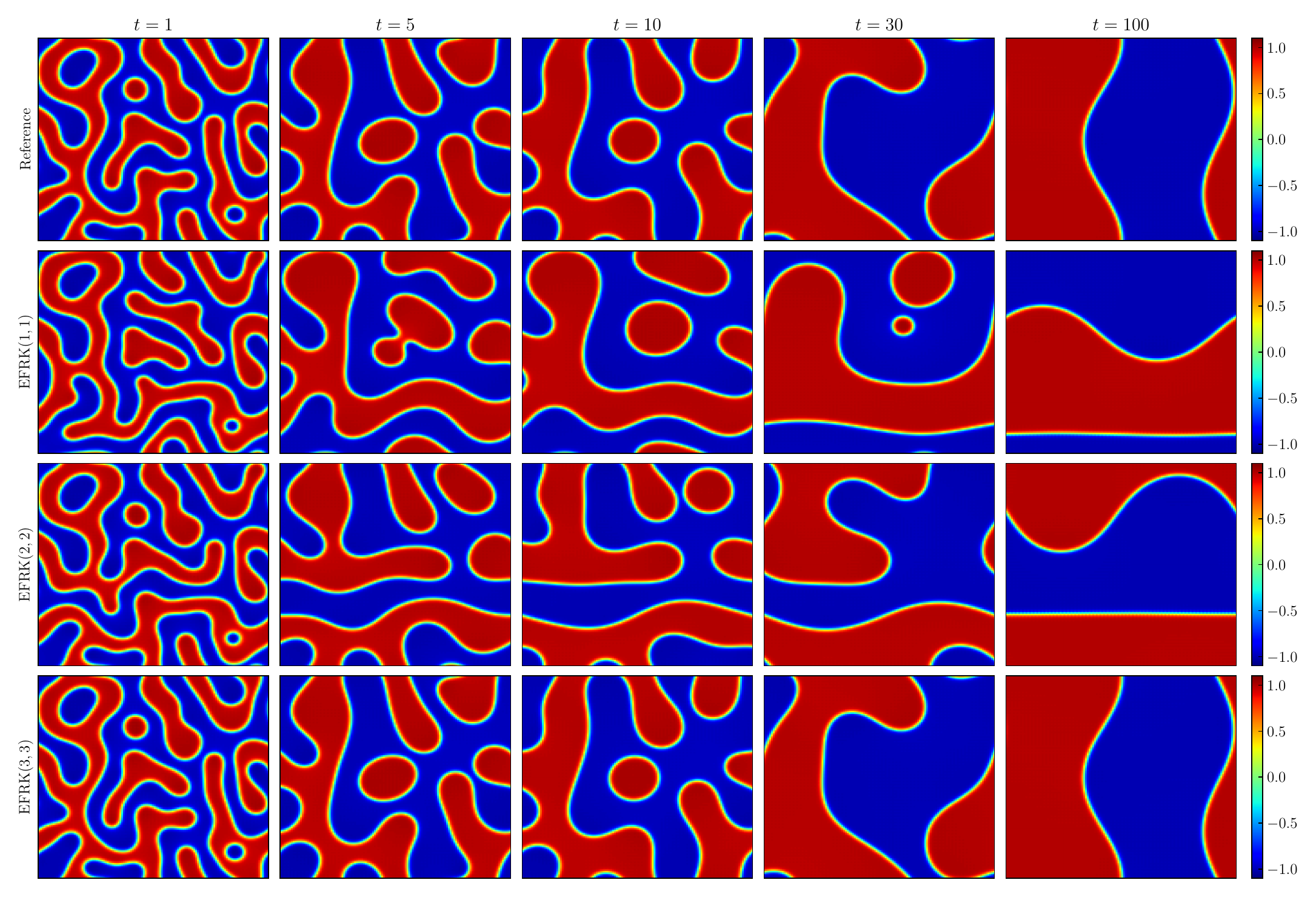}
	\caption{Example \ref{exm:advantage}: Comparison of solutions computed by EFRK(1,1), EFRK(2,2) and EFRK(3,3) with same time step size $\tau = 10^{-3}$. Parameters: $\epsilon^2 = 0.0025$ and $\kappa = 2.0$.}
	\label{fig:advantageSnapshot}
\end{figure}

\begin{figure}[htbp]
	\centering
	\includegraphics[width=0.9\linewidth]{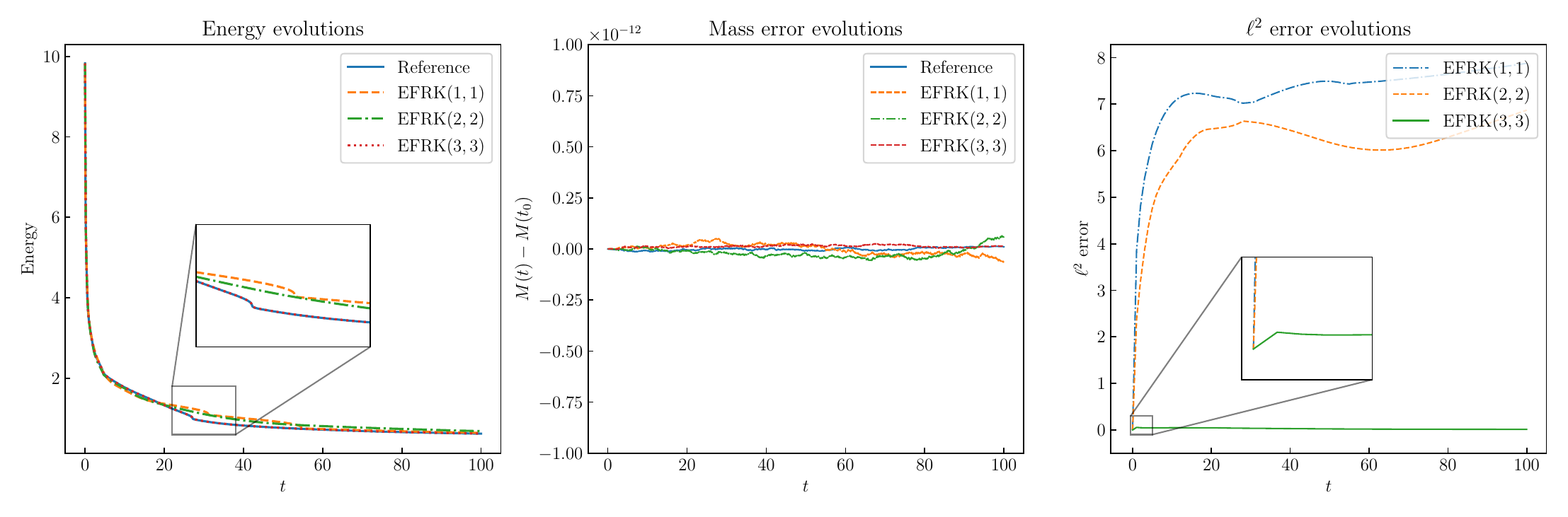}
	\caption{Example \ref{exm:advantage}: Evolution of the energy, change in the mass and $\ell^2$-error computed by EFRK(1,1), EFRK(2,2) and EFRK(3,3) with same time step size $\tau = 10^{-3}$. Parameters: $\epsilon^2 = 0.0025$ and $\kappa = 2.0$.} 
	\label{fig:advantageProperty}
\end{figure}

At first, we simulate the coarsening process of the CH equation using EFRK schemes with varying order: EFRK(1,1), EFRK(2,2), and EFRK(3,3) scheme, and utilize a constant time step size $\tau =  10^{-3}$. Then, we present the solution snapshots in Fig. \ref{fig:advantageSnapshot} and the evolution of energy curves, the change of mass, and $\ell^2$ error in Fig. \ref{fig:advantageProperty}. The results show that all schemes successfully maintain energy dissipation and mass conservation. However, upon comparison with the reference solutions, significant differences are apparent in the evolution of results generated by the EFRK(1,1) and EFRK(2,2) schemes. These discrepancies are also reflected in the energy and $\ell^2$ error diagram. Conversely, the solution generated by the EFRK(3,3) scheme closely matches the reference solution, with its energy curve almost overlapping the reference curve.

Subsequently, we simulate the coarsening process of the CH equation again using a finer time step for the low-order EFRK scheme. Specifically, we use $\tau = 10^{-3}/16$ for the EFRK(1,1) scheme and $\tau = 10^{-3}/2$ for the EFRK(2,2) scheme. The numerical are present in Fig. \ref{fig:advantageSnapshot_2} and \ref{fig:advantageProperty_2}, respectively. At this time, we observe that all solution snapshots are consistent with the reference solution, which confirms the reliability of the reference solution and indicates that all the methods converge to the same solution as $\tau \rightarrow 0$. In addition, from the $\ell^2$ error diagram, we see that although the results are similar, the EFRK(3,3) scheme still achieves the highest accuracy with a larger step size. These findings emphasize the advantages of employing higher-order schemes.
 
Moreover, by comparing the energy and $\ell^2$ error evolution curves of EFRK($3,3$), we observe that the error marked increases during the periods of energy rapid decrease. As the energy towards a local minimum, the corresponding numerical solution varies slowly, resulting in a stable error. This implies that it is crucial to utilize smaller time steps during rapid energy variations to maintain accuracy. Conversely, during the energy change slowly, we can increase the time step without compromising accuracy. Consequently, strategically adjusting the time step size allows us to effectively balance accuracy and computational efficiency.

\begin{figure}[htbp]
	\centering
	\includegraphics[width=0.9\linewidth]{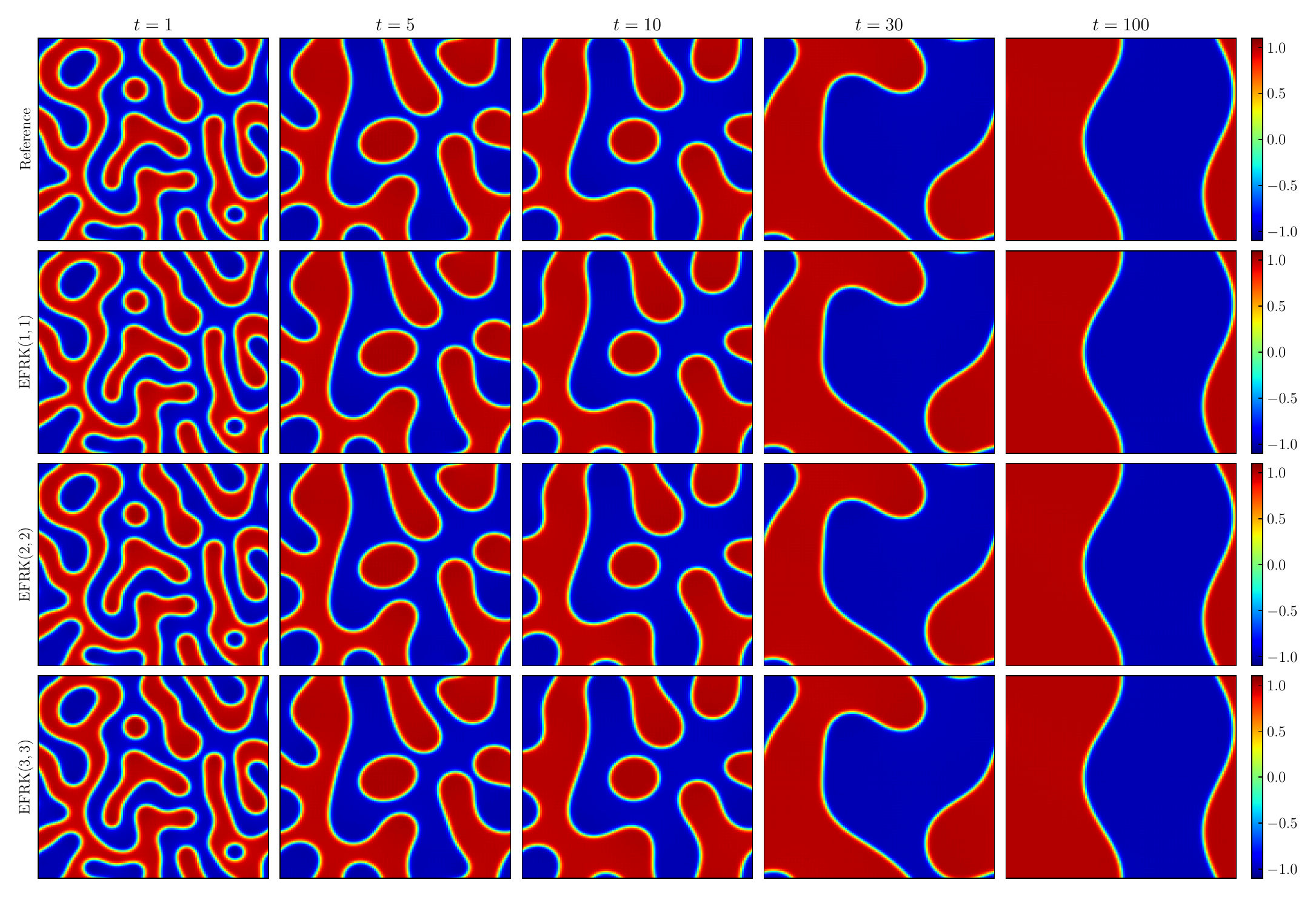}
	\caption{Example \ref{exm:advantage}: Comparison of solutions computed by EFRK(1,1) with $\tau = 10^{-3}/16$, EFRK(2,2) with $\tau = 10^{-3}/2$ and EFRK(3,3) with $\tau = 10^{-3}$. Parameters: $\epsilon^2 = 0.025$, and $\kappa = 2.0$.}
	\label{fig:advantageSnapshot_2}
\end{figure}

\begin{figure}[htbp]
	\centering
	\includegraphics[width=0.9\linewidth]{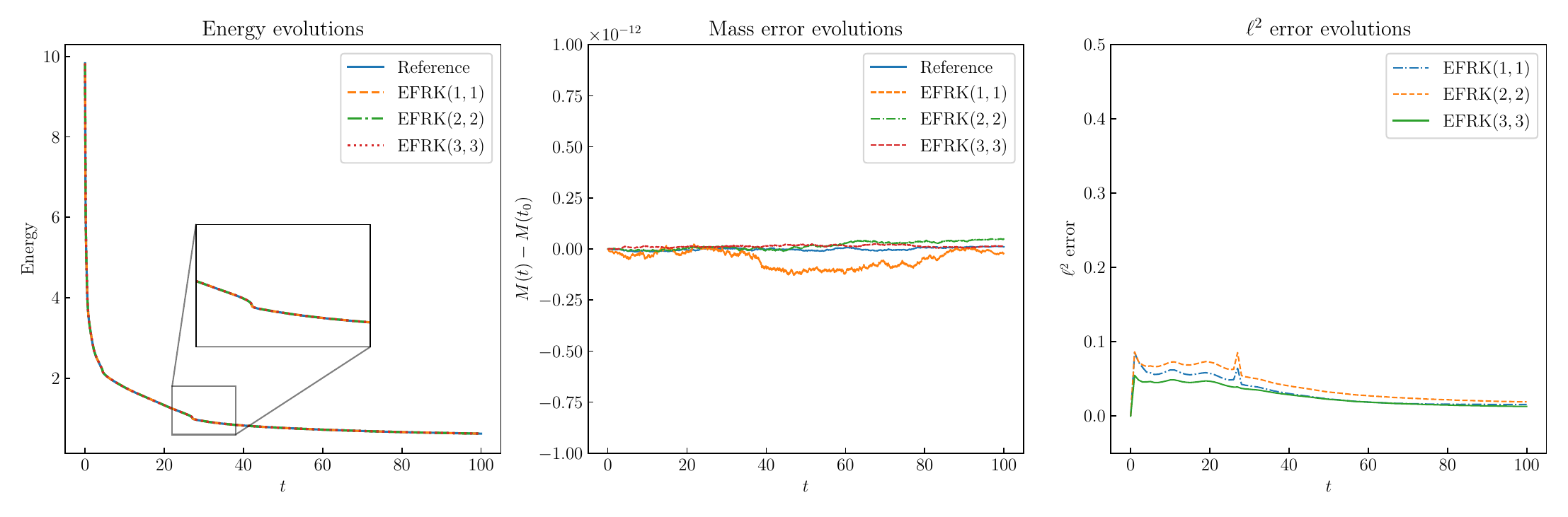}
	\caption{Example \ref{exm:advantage}: Evolution of the energy, change in the mass and $\ell^2$-error computed by EFRK(1,1) with $\tau = 10^{-3}/16$, EFRK(2,2) with $\tau = 10^{-3}/2$ and EFRK(3,3) with $\tau = 10^{-3}$. Parameters: $\epsilon^2 = 0.025$, and $\kappa = 2.0$.} 
	\label{fig:advantageProperty_2}
\end{figure}

\subsection{Adaptive time stepping}
\begin{exm}
\label{exm:adaptive}
The energy of the Cahn-Hilliard equation experiences a significant initial decrease due to nonlinear interactions, followed by a gradual descent until it reaches a local minimum. This behavior indicates that the solution undergoes rapid variations in a very short time, followed by slower changes. As mentioned in previous experiments, implementing an adaptive time-stepping algorithm is appropriate when the time step solely controls the accuracy, as opposed to schemes that alter the steady state of the original system.

For the Cahn--Hilliard equation, there are several adaptive time-stepping strategies. In this study, we adopt the strategy proposed by Z. Zhang and Z. Qiao \cite{zhang2012adaptive}, which adjusts the time step based on the magnitude of energy variation. The formula for the adaptive time step is:
\begin{equation*}
	\tau = \max\left(\tau_{\min},\frac{\tau_{\max}}{\sqrt{1+ \alpha|E'(t)|^2}}\right),
\end{equation*}
where $E(\cdot)$ is the energy functional of the model that we proposed, $\tau_{\min}$ and $\tau_{\max}$ are the pre-set minimum and maximum bounds of the time steps to ensure that they fall within a reasonable range, and the constant $\alpha$ is used to control the level of adaptivity. Under the same setting as in the previous experiments, we take the EFRK(3,3) scheme as an example to demonstrate the performance of the adaptive time-stepping strategy. In our numerical experiment, we set $\alpha = 100$, and the minimum and maximum time steps are selected to be $10^{-5}$ and $10^{-2}$, respectively. The initial step is taken as the minimum time step. For comparison, we compute EFRK$(3,3)$ solutions using a small uniform time step $\tau = 10^{-5}$ and a large uniform time step $\tau = 10^{-2}$ as references. 
\end{exm}

Fig. \ref{fig:adaptSnpshot} and Fig. \ref{fig:adaptDiag} display the solution snapshots, energy evolution, variation of the time step, and CPU time consumption for solving the problem using different time stepping strategies. From the result, we can conclude that although large time steps result in energy stability and the shortest computation time, the obtained topology is vastly dissimilar to that obtained with small time steps. On the other hand, the solution obtained with adaptive time steps coincides with the reference solution. Fig. \ref{fig:adaptDiag} clearly demonstrates that large time steps lack accuracy in capturing drastic energy changes, whereas adaptive time steps can adjust the step size to align with energy changes, leading to solutions very close to the reference solution obtained with smaller time steps. Furthermore, apart from the initial period, the adaptive time step typically stays in the range of $10^{-3}$ and $10^{-2}$, which exhibits comparable computational efficiency to the application of large time steps. These results underscore the efficacy and accuracy of utilizing adaptive time steps in computational simulation.

\begin{figure}[htbp]
	\centering
	\includegraphics[width=0.95\linewidth]{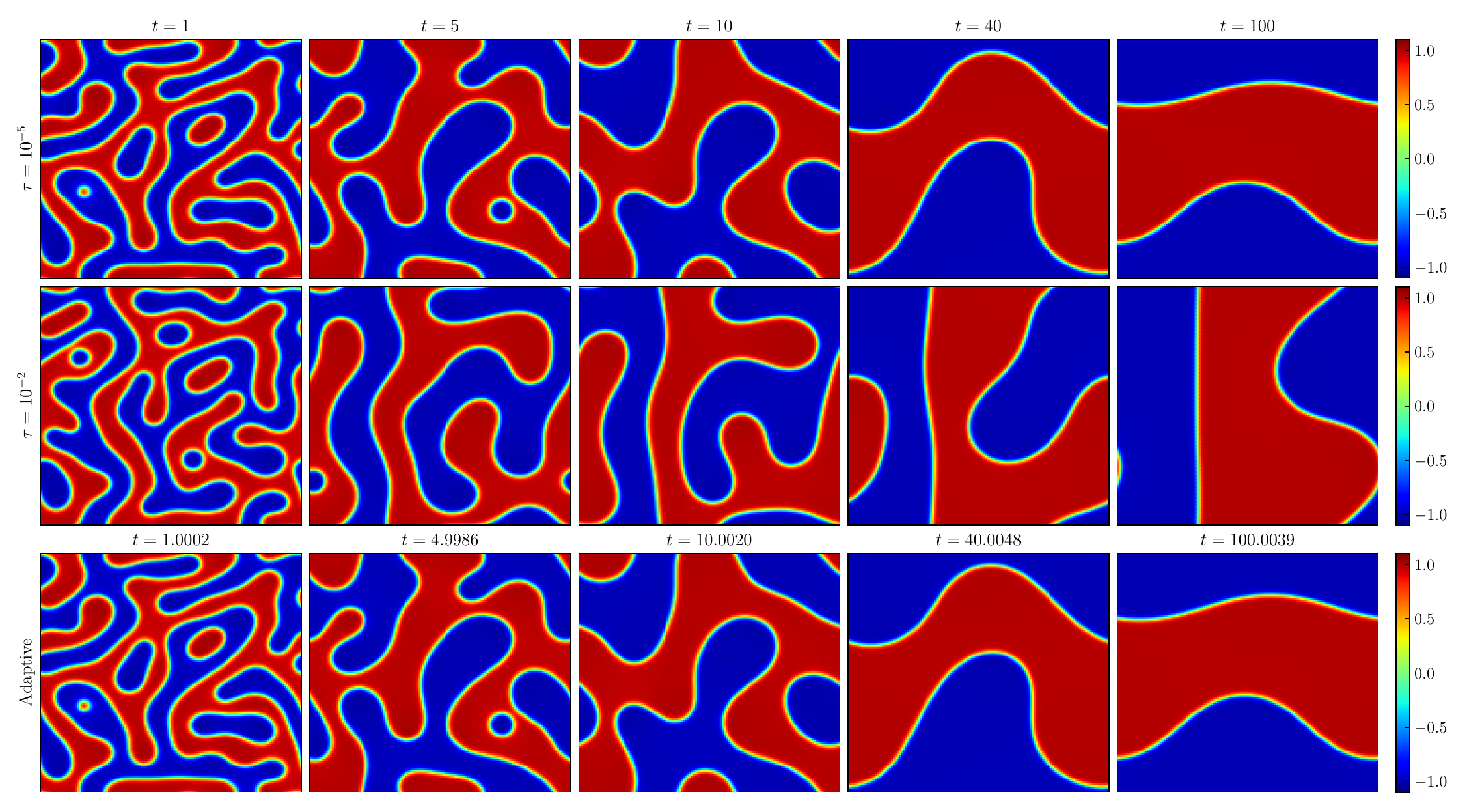}
	\caption{Example \ref{exm:adaptive}: Solution snapshots of the Cahn--Hilliard equation obtained by EFRK(3,3) with different time stepping strategies.}
	\label{fig:adaptSnpshot}
\end{figure}

\begin{figure}[htbp]
	\centering
	\includegraphics[width=0.95\linewidth]{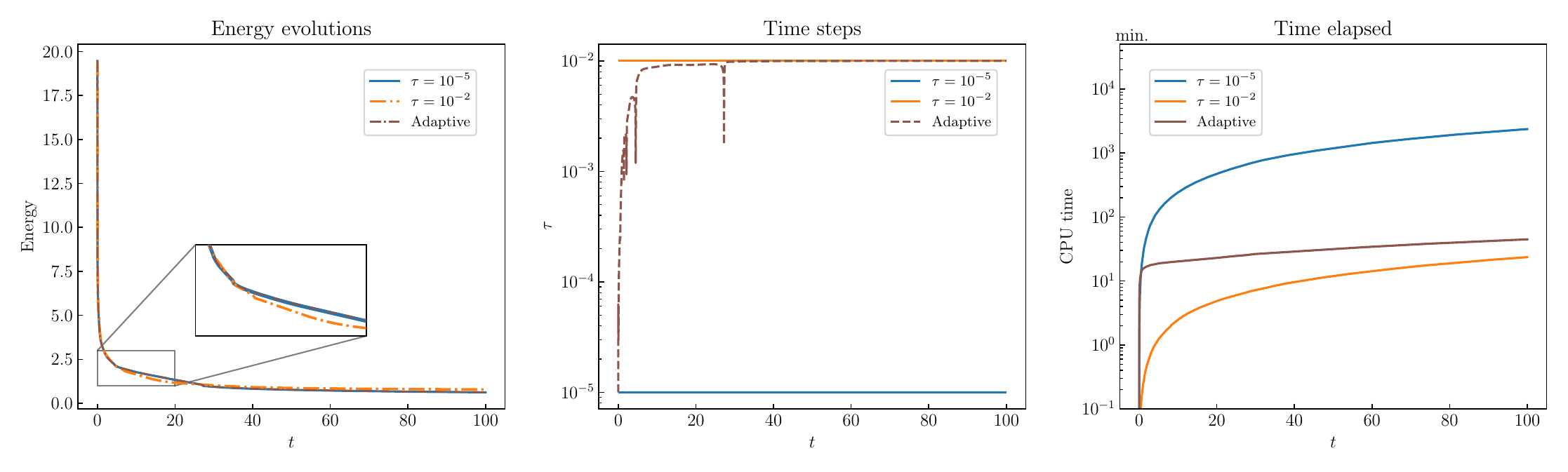}
	\caption{Example \ref{exm:adaptive}: The energy evolution, time step variation, and CPU time elapse along the time with different time step strategies.}
	\label{fig:adaptDiag}
\end{figure}

\begin{exm}
\label{exm:advantage2}
To highlight the superiority of the EFRK scheme, we conducted simulations of 2D coarsening dynamics and compared the results with those yielded by the IFRK schemes, which, as previously discussed, do not preserve the equilibrium state. We set the coefficient $\epsilon^2 = 0.002$ and consider the domain $\Omega = (0,2\pi)^2$ with $N_1 = N_2 = 128$. The initial condition is generated by random values uniformly distributed from $-0.5$ to $0.5$. We utilize the adaptive time strategy, employing the same settings as in the previous experiment. The reference solution is computed with $\tau = 10^{-5}$ using the ERK(3,3) scheme with $\kappa = 0$.
\end{exm}

Fig. \ref{fig:coasrsSnapshot} presents the snapshots of the numerical solutions at $T = 1,4,10,30,50$ computed by EFRK$(2,2)$, EFRK$(3,3)$, and IFRK$(3,3)$, respectively. Despite identical settings, all solutions only exhibit similar topology at the initial stages of evolution. Nevertheless, the solutions obtained by the EFRK schemes remain consistent with the reference solutions throughout the entire time duration.

Based on Fig. \ref{fig:coasrsProperty}, we can explain the reasons for the discrepancies in the numerical solutions. It can be observed that while all schemes conserve mass, the EFRK(2,2) and EFRK(3,3) schemes, thanks to their unconditional energy stability and preservation of equilibrium, maintain decreasing energy curves that align with the reference even during intervals of drastic change. However, the equilibrium of the IFRK(3, 3) is influenced by the time-step size. When the time step varies, the equilibrium also change, leading to fluctuations in the energy. This issue causes discrepancies between the behavior of IFRK(3, 3) and the reference solutions, thereby resulting in relatively unreliable numerical solutions.

\begin{figure}[htbp]
	\centering
	\includegraphics[width=0.95\linewidth]{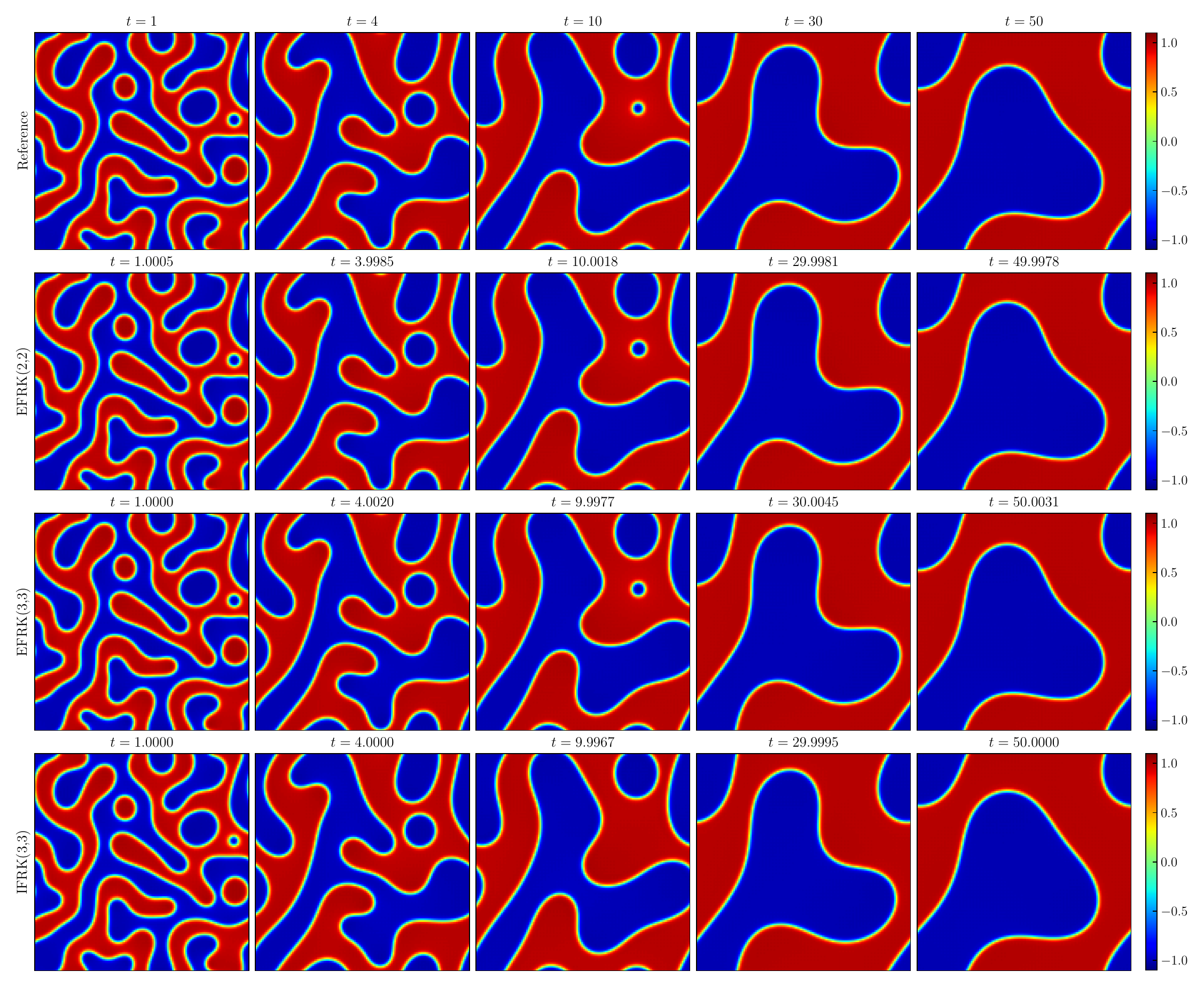}
	\caption{Example \ref{exm:advantage2}: solution for Cahn--Hilliard equation computed using EFRK$(2,2)$, EFRK$(3,3)$ and IFRK$(3,3)$ schemes by the adaptive time stepping strategy.}
	\label{fig:coasrsSnapshot}
\end{figure}

\begin{figure}[htbp]
	\centering
	\includegraphics[width=0.9\linewidth]{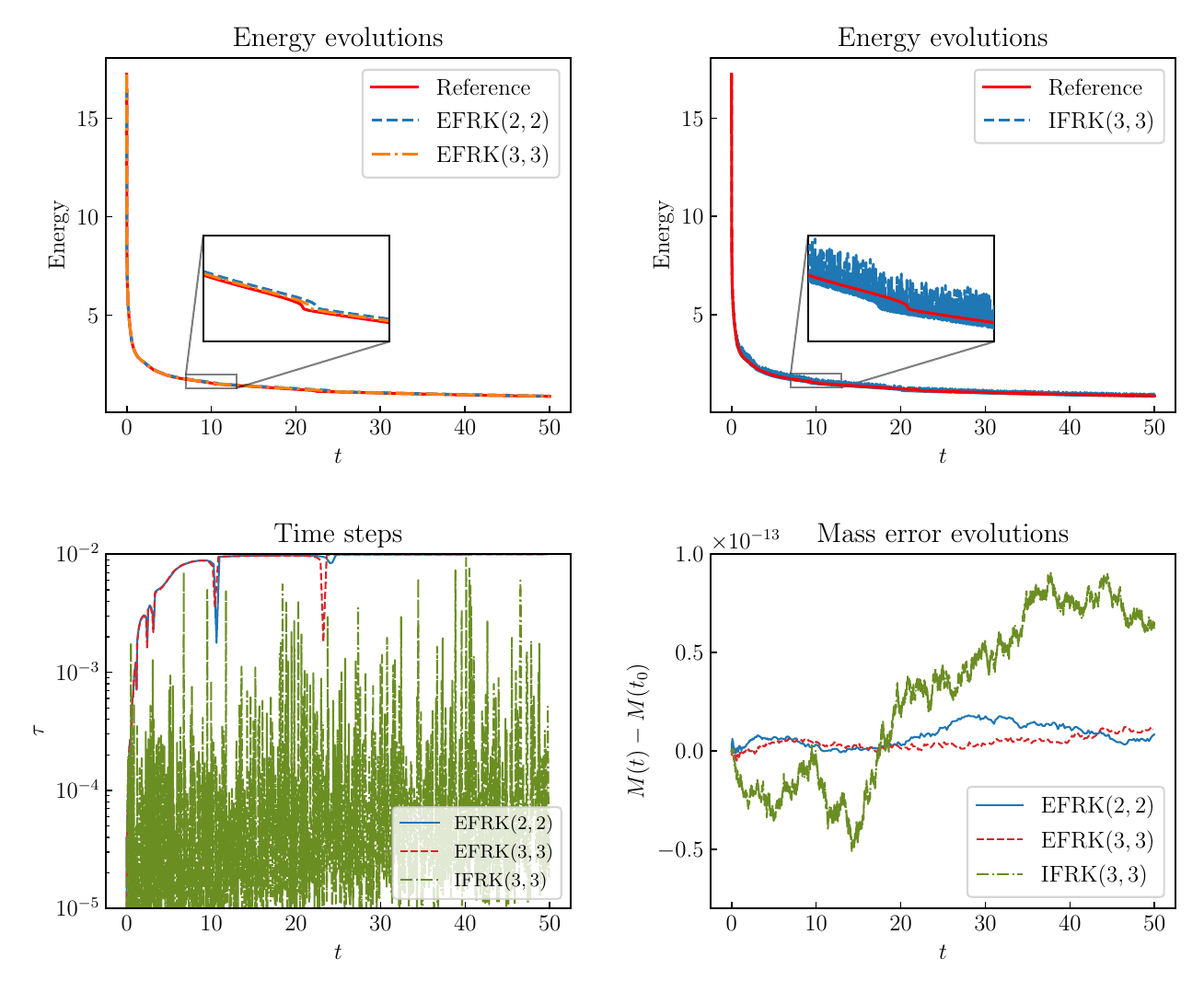}
	\caption{Example \ref{exm:advantage2}: evolution of energy $E$ and the mass computed by EFRK$(2,2)$, EFRK$(3,3)$ and IFRK$(3,3)$ schemes, the time step variation is determined by the adaptive time stepping strategy.}
	\label{fig:coasrsProperty}
\end{figure}

\section{Conclusion}
\label{sec:7}

In this work, we have presented a novel framework for constructing up to third-order unconditionally energy-stable schemes to solve the CH equation. By utilizing Taylor-type polynomial approximations, our approach eliminates the exponential damping effect in the IFRK approach. Consequently, the resulting exponential-free Runge--Kutta (EFRK) framework not only preserves the equilibrium of the problem but also improves accuracy for long-time simulations. We also provide a unified proof to demonstrate that the proposed EFRK schemes can unconditionally preserve energy dissipation of the CH equation. In addition, using linear stability analysis, we demonstrated that the EFRK schemes exhibit $A$-stability for any time step when an appropriate stabilization parameter is selected. Numerous experiments substantiate the rationality and accuracy of the Taylor polynomial approximations. In addition, it is worth noting that the EFRK framework and its associated results can be extended to other phase field models as well, such as the epitaxial growth model \cite{herman2012molecular} and phase field crystal model \cite{elder2002modeling}.

Two distinctive contributions of this work are the improved accuracy of the proposed EFRK schemes compared to traditional IFRK schemes for long-time simulations and the unified strategy to analyze energy stability. However, a limitation of the current work is that the Taylor polynomial approximations fail to construct fourth-order EFRK schemes. To develop energy-stable schemes with higher temporal order, it may be worthwhile to establish a more general strategy to approximate exponential functions and analyze the energy stability.

\section*{Acknowledge}
The authors would like to thank the anonymous referees for the valuable comments and constructive suggestions that have led to significant improvements in this work.

This work was supported by the National Natural Science Foundation of China (12271523, 12071481, 12371374),  National Key Research and Development Program of China (2020YFA0709803), Science and Technology Innovation Program of Hunan Province (2022RC1192, 2021RC3082),  Natural Science Foundation of Hunan (2021JJ20053), Defense Science Foundation of China (2021-JCJQ-JJ-0538, 2022-JCJQ-JJ-0879), and National University of Defense Technology (2023-LXY-FHJJ-002).

\section*{Declarations}
The authors declare that they have no known competing financial interests or personal relationships that could have appeared to influence the work reported in this paper. The code and data accompanying this manuscript are publicly available at \href{https://github.com/HaifengWang1031/EFRK-CH}{https://github.com/HaifengWang1031/EFRK-CH}.

\bibliographystyle{unsrt}
\bibliography{references}

\end{document}